\documentclass{article}
\usepackage[utf8]{inputenc}
\usepackage{cite}
\usepackage[svgnames,dvipsnames]{xcolor}
\usepackage[a4paper, total={6in, 8.5in}]{geometry}
\usepackage{graphicx}
\usepackage{amssymb}
\usepackage{mathtools}
\usepackage{tikz-cd}
\usepackage{booktabs}
\usepackage{calc}
\usepackage{titling}
\usepackage{amsthm}
\usepackage{amscd}
\usepackage{caption}
\usepackage{amsmath}
\usepackage{MnSymbol}
\usepackage{tikz}
\usepackage{dsfont}
\usepackage{sectsty}
\usepackage{xparse}

\sectionfont{\fontsize{12}{15}\selectfont}

\newcommand{\R}{\mathbb{R}}
\newcommand{\C}{\mathbb{C}}
\newcommand{\Z}{\mathbb{Z}}

\newtheorem{prop}{Proposition}
\newtheorem{theorem}{Theorem}
\newtheorem{lemma}{Lemma}

\newtheorem{corollary}{Corollary}

\title{Nilpotency in instanton homology, and the\\ framed instanton homology of a surface times a circle}
\author{William Chen \& Christopher Scaduto}
\date{}

\begin{document}

\maketitle

\begin{abstract}
In the description of the instanton Floer homology of a surface times a circle due to Mu\~{n}oz, we compute the nilpotency degree of the endomorphism $u^2-64$. We then compute the framed instanton homology of a surface times a circle with non-trivial bundle, which is closely related to the kernel of $u^2-64$. We discuss these results in the context of the moduli space of stable rank two holomorphic bundles with fixed odd determinant over a Riemann surface. 
\end{abstract}

\vspace{.6cm}

\section{Introduction}

For a closed, oriented and connected 3-manifold equipped with an $SO(3)$-bundle that restricts non-trivially to some embedded, oriented surface, Floer \cite{floer-dehn} defined a relatively $\Z/8$-graded complex vector space called {\emph{instanton homology}}. We call such bundles {\emph{non-trivial admissible}}. Floer also earlier defined his instanton homology for integral homology 3-spheres, in which case the bundle is trivial \cite{floer-instanton}. When the bundle is non-trivial, the instanton homology comes with a degree four endomorphism $u$, which is an isomorphism. In fact, the degree zero endomorphism $u^2-64$ is nilpotent, as exhibited in the work of Fr\o yshov \cite{froyshov}. Although the expression $u^2-64$ is in general not well-defined on Floer's instanton homology in the homology 3-sphere case, it does make sense in the framework of Fr\o yshov's {\emph{reduced}} instanton homology groups in loc. cit., and Fr\o yshov proves that $u^2-64$ is nilpotent in this context as well.

The instanton homology of a surface times a circle with bundle whose second Stiefel-Whitney class is Poincar\'{e} dual to the circle plays an important r\^{o}le in the general structure of instanton homology. It is equipped with a ring structure, which was computed by Mu\~{n}oz \cite{munoz-ring}. Using Mu\~{n}oz's ring one can also see the nilpotency of $u^2-64$ for non-trivial admissible bundles. 

Our first result is on the degree of this nilpotency. Let $\Sigma$ be a closed, oriented surface of genus $g\geqslant 1$, and denote Mu\~{n}oz's instanton homology with non-trivial bundle as above by $I(\Sigma\times S^1)_w$.\\

\begin{theorem}\label{prop:main}
The endomorphism $u^2-64$ acting on the instanton homology $I(\Sigma\times S^1)_w$ satisfies
\[
    {\emph{\text{min}}}\Big\{ {n \geqslant 1 }: \;\;(u^2-64)^n =0 \Big\} \;\;=\;\;  2\left\lceil  g/2 \right\rceil -1.
\]
\end{theorem}
\vspace{.55cm}

\noindent This in turn gives an upper bound for the nilpotency degree of $u^2-64$ on the instanton homology of an arbitrary 3-manifold with non-trivial admissible bundle. Following the notation of \cite{km-sutures}, suppose $Y$ is a closed, oriented and connected 3-manifold with a Hermitian line bundle $w\longrightarrow Y$ whose first chern class has odd pairing with some closed, oriented surface $\Sigma\subset Y$. The class $w_2(w)$ determines the $SO(3)$-bundle used to define the instanton homology $I(Y)_w$. In the sequel we call the pair $(Y,w)$ {\emph{non-trivial admissible}}. In the cylinder $Y\times \R$, take a $\Sigma\times D^2$ neighborhood of the surface $\Sigma\times \{0\}$. By a stretching argument employed by Fr{\o}yshov, the $u$-map of $I(Y)_w$ factors through $I(S^1\times \Sigma)_w$ and we have, cf. \cite[Thm. 9]{froyshov}:\\

\begin{corollary}
Let $(Y,w)$ be a non-trivial admissible pair. If there is a closed oriented surface $\Sigma\subset Y$ of genus zero with $w|_\Sigma$ nontrivial, then $I(Y)_w=0$. Otherwise, $u^2-64$ acting on $I(Y)_w$ satisfies

\vspace{.10cm}
\begin{equation*}
    {\emph{\text{min}}}\Big\{ {n \geqslant 1 }: \;\;(u^2-64)^n =0 \Big\} \;\; \leqslant \;\; 2\left\lceil  g/2 \right\rceil -1 \label{eq:cor1}
\end{equation*}
\vspace{-.06cm}

\noindent where $g\geqslant 1$ is the minimal genus over all closed oriented surfaces $\Sigma\subset Y$ with $w|_\Sigma$ non-trivial.
\end{corollary}
\vspace{.5cm}

\noindent One may go on to deduce results regarding the instanton homology of homology 3-spheres from Theorem \ref{prop:main} and Floer's exact triangle, but, as we remark in the closing of this introduction, the full power of the above nilpotency degree seems most relevant for non-trivial bundles.

Our second result regards the kernel, or more precisely the mapping cone, of the endomorphism $u^2-64$ on Mu\~{n}oz's ring. This is related to the {\emph{framed instanton homology}} of a surface times a circle. For a general (not necessarily non-trivial admissible) pair $(Y,w)$ in which $w_2(w)$ determines an $SO(3)$-bundle over $Y$, the framed instanton homology $I^\#(Y)_w$ is a $\Z/4$-graded vector space. For a non-trivial admissible pair $(Y,w)$ there is a long exact sequence

\vspace{.05cm}

\begin{equation}
\begin{tikzcd}[column sep=large]
   \cdots \arrow{r} &  I^\#(Y)_w \arrow{r} & I(Y)'_w \arrow{r}{u^2-64} & I(Y)'_w \arrow{r} & I^\#(Y)_w \arrow{r} & \cdots \label{eq:twistedgysin}
\end{tikzcd}
\end{equation}
\vspace{.15cm}

\noindent Here we remark that for a non-trivial admissible pair $(Y,w)$ the relatively $\Z/8$-graded vector space $I(Y)_w$ is 4-periodic, and $I(Y)_w'$ denotes its quotient by a degree four involution. The framed instanton homology $I^\#(Y)_w$ is defined to be $I(Y\# T^3)'_v$ in which $v$ restricts over $Y$ to $w$, and is non-trivial over the 3-torus $T^3$. This invariant was introduced by Kronheimer and Mrowka \cite{km-knot}. The exact sequence (\ref{eq:twistedgysin}) is a restatement of \cite[Thm 1.3]{scaduto} and is an application of the connect sum theorem of Fukaya \cite{fukaya}.

Via the exact sequence (\ref{eq:twistedgysin}) and Mu\~{n}oz's calculation of the ring $I(\Sigma\times S^1)_w$ we compute the framed instanton homology of a surface times a circle, with the same bundle as above. We define the integer $s_i(g)$ to be the sum $\sum {2g\choose k}$ over the indices $k$ satisfying $0\leqslant k < g$ and $k\equiv i\text{ (mod 4)}$. We write $b_i$ for the dimension of the $i$ (mod 4) graded summand of $I^\#(\Sigma\times S^1)_w$.\\
\vspace{.25cm}

\begin{theorem}\label{thm:main}
Let $g\geqslant 1$, and define $\varepsilon = \varepsilon(g)\in \{0,1\}$ by letting $\varepsilon =1$ if $g$ is odd and $\varepsilon = 0$ if $g$ is even. Then the betti numbers of the $\Z/4$-graded vector space $I^\#(\Sigma\times S^1)_w$ are as follows:
\vspace{.25cm}
\[
    b_{0+\varepsilon}\; =\; b_{1+\varepsilon} \;=\; \frac{g+1}{2}{2g\choose g} - 2^{g-2}(1+1\cdot 2^{g-1}) - s_{1-\varepsilon}(g),
\]

\vspace{.35cm}

\[
    b_{2+\varepsilon}\;=\;b_{3+\varepsilon}\; =\; \frac{g+1}{2}{2g\choose g} - 2^{g-2}(1+3\cdot 2^{g-1}) + s_{1-\varepsilon}(g).
\]
\end{theorem}
\vspace{.65cm}

\noindent The $\Z/4$-grading used for the above theorem is subject to an invertible $\Z/4$-linear transformation depending on one's conventions. We first use the $\Z/4$-grading of Mu\~{n}oz's ring, and the convention that all maps in (\ref{eq:twistedgysin}) are degree zero except for $I^\#(Y)_w\longrightarrow I(Y)'_w$ which is of degree $1$ (mod 4). 

\setlength{\tabcolsep}{10pt}

\begin {table}
\begin{center}
\caption {Mod 4 graded betti numbers for $I^\#(\Sigma\times S^1)_w$} \label{fig:table1} 
\begin{tabular}{ r|ccccccccc } 
\toprule
genus $g$  & $1$ & $2$ & $3$ & $4$ &  $5$ & $6$ & $7$ & $8$\\
\midrule
$b_{0+\varepsilon} = b_{1+\varepsilon} $ & 0 &  2  &  29  & 131   & 409 & 1,902 & 10,646 & 45,275\\
$b_{2+\varepsilon} = b_{3+\varepsilon} $  & 1 &  6  &  15  & 88  & 575 & 2,486 & 8,554  & 37,659\\ 
\midrule
Total rank  &  2  &  16 & 88 & 428 & 1,968 & 8,776 & 38,400 & 165,868\\
\bottomrule
\end{tabular}
\end{center}
\vspace{.25cm}
\end{table}

\begin {table}
\begin{center}
\caption {Mod 4 graded betti numbers for $H_\ast(N^g_0\sqcup N_0^g)$} \label{fig:table2} 
\begin{tabular}{ r|ccccccccc } 
\toprule
genus $g$ \;\; & $1$ & $2$ & $3$ & $4$ &  $5$ & $6$ & $7$ & $8$\\
\midrule
$n_{0+\varepsilon} = n_{1+\varepsilon}$  & 0 & 2 & 44  &  188  & 464 & 2,188 & 14,104 & 59,096 \\
$n_{2+\varepsilon} = n_{3+\varepsilon}$  & 2 & 10  & 16  &  92 & 796 & 3,356 & 9,920  & 43,864 \\ 
\midrule
Total rank  & 4  &  24 & 120  & 560 & 2,520 & 11,088 &  48,048  & 205,920 \\
\bottomrule
\end{tabular}
\end{center}
%\caption*{More info goes here...}
\vspace{.25cm}
\end{table}

The euler characteristic of the instanton homology in Theorem \ref{thm:main} is seen to be zero. Indeed, the euler characteristic of the framed instanton homology of any pair $(Y,w)$ with $b_1(Y)>0$ is zero \cite[Cor 1.4]{scaduto}. Also, the addition of the betti numbers yields the total dimension:\\

\begin{equation}
    \dim I^\#(\Sigma\times S^1)_w \;=\; b_0+b_1+b_2+b_3 \;=\; 2(g+1){2g\choose g} -2^{g}(1+2^{g}).\label{eq:totaldim}
\end{equation}
\vspace{.30cm}

\noindent

The framed instanton homology $I^\#(\Sigma\times S^1)_w$ is roughly the Morse homology of a Chern-Simons functional whose critical set may be identifed with two copies of the framed moduli space $N^g_0$ of flat $SO(3)$-connections over $\Sigma$ with non-trivial $w_2$. A theorem of Fukaya \cite{fukaya} says that when certain transversality conditions are met, there is a Bott-Morse type spectral sequence whose $E^1$-page is the singular homology of $N^g_0\sqcup N^g_0$ that converges to $I^\#(\Sigma\times S^1)_w$. We write this schematically as

\vspace{.35cm}
\begin{equation}
    E^1 =H_\ast(N_0^g\sqcup N_0^g)\;\;\; \rightrightarrows\;\;\; I^\#(\Sigma\times S^1)_w.\label{eq:specseq}
\end{equation}
\vspace{.15cm}

\noindent We remark that we do not actually verify all the hypotheses of Fukaya's construction to produce (\ref{eq:specseq}), but our results suggest that there is such a spectral sequence. Indeed, a comparison of the mod 4 graded betti numbers of the $E^\infty$- and $E^1$-pages in Tables \ref{fig:table1} and \ref{fig:table2}, respectively, shows the requisite rank inequalities, and the presence of non-trivial differentials on the $E^1$-page. We have written $n_i$ for the $i^\text{th}$ betti number of $H_\ast(N_0^g\sqcup N_0^g)$ with $\Z/4$-grading induced by mod 4 reduction. In particular, $H_\ast(N_0^g\sqcup N_0^g)$ has total dimension given by $2g{2g \choose g}$, which may be compared with the smaller number (\ref{eq:totaldim}). This apparent non-collapsing of the spectral sequence is in contrast with the unframed case. For further explanation and some background, see Section \ref{sec:background}.

The computation of Theorem \ref{thm:main} is clarified by considering the ``invariant part'' of the instanton homology. There is an action of the diffeomorphism group of the surface $\Sigma$ on the instanton homology $I(\Sigma\times S^1)'_w$ which factors through an action of $\text{Sp}(2g,\Z)$. We write $I_\text{inv}(\Sigma\times S^1)'_w$ for the subspace on which $\text{Sp}(2g,\Z)$ acts trivially. The decomposition of the instanton homology into irreducible $\text{Sp}(2g,\Z)$-representations has the following convenient property: the summands corresponding to non-trivial representations may be understood in terms of instanton groups of lower genus, and thus by induction it suffices to compute the invariant part. This reduction to the $\text{Sp}(2g,\Z)$-invariant part goes back to a much studied recursive presentation for the cohomology ring of stable rank two holomorphic bundles over a Riemann surface with fixed odd determinant, which is mentioned in Section \ref{sec:ring}. In this paper, we define the invariant part of the framed instanton homology, denoted $I_\text{inv}^\#(\Sigma\times S^1)_w$, to be the homology mapping cone of $u^2-64$ acting on $I_\text{inv}(\Sigma\times S^1)'_w$. We find that the dimension of this vector space has the simple expression

\vspace{.25cm}
\begin{equation}
    \dim I_\text{inv}^\#(\Sigma\times S^1)_w \;=\; g(g+1) + 2\cdot\delta_{g+2}^4\label{eq:invdim}
\end{equation}
\vspace{.10cm}

\noindent in which $\delta^4_i$ is equal to $1$ if $i \equiv 0$ (mod 4) and is otherwise zero. Next, we mention that computing the mapping cone of the endomorphism $u^2-64$ can be done by computing the mapping cones of $u+8$ and $u-8$ separately, and then summing. More precisely, for any non-trivial admissible pair $(Y,w)$ we can decompose the framed instanton homology into two pieces:

\vspace{.22cm}
\[
    I^\#(Y)_w \; = \; I^\#(Y)^+_w\oplus I^\#(Y)^-_w
\]
\vspace{.10cm}

\noindent in which $I^\#(Y)^\pm_w$ is isomorphic to the homology of the mapping cone of the endomorphism $u\pm 8$ acting on $I(Y)_w$. Thus we have two separate Gysin-like sequences:

\vspace{.05cm}

\begin{equation*}
\begin{tikzcd}[column sep=large]
   \cdots \arrow{r} &  I^\#(Y)^\pm_w \arrow{r} & I(Y)'_w \arrow{r}{u\pm 8} & I(Y)'_w \arrow{r} & I^\#(Y)^\pm_w \arrow{r} & \cdots 
\end{tikzcd}
\end{equation*}
\vspace{.10cm}

\noindent We compute the dimensions of the ``invariant'' parts of the vector spaces $I^\#(\Sigma\times S^1)^\pm_w$ separately, from which we will deduce (\ref{eq:invdim}). This is all done using Mu\~{n}oz's recursive presentation for the instanton Floer ring $I(\Sigma\times S^1)_w'$.
\\

\vspace{.75cm}

\noindent \textbf{Further discussion.} Relations in the instanton homology of a surface times a circle have been used to prove adjunction inequalities \cite{munoz-higher}, the simple-type conjecture for Donaldson invariants \cite{munoz-ff,froyshov}, and an inequality for Fr\o yshov's $h$-invariant \cite{froyshov-inequality}. In most of these situations, some condition on simple-connectivity or the vanishing of the first homology group of a 4-manifold with boundary implies that the element $\gamma$ in Mu\~{n}oz's ring (see Section \ref{sec:ring}) is effectively zero. Mu\~{n}oz observed that in the instanton homology of a genus $g$ surface times a circle, one always has 

\vspace{.20cm}

\[
    \left(u^2-64\right)^{\lceil g/2 \rceil} \; \in \; \text{im}\left(\gamma\right),
\]
\vspace{.15cm}

\noindent which means that in these situations the nilpotency degree of $u^2-64$ is (effectively) about half that of the general nilpotency degree that we exhibit here. The relation of Theorem \ref{prop:main} may be more relevant in situations with non-trivial first homology.\\

\vspace{.55cm}

\noindent \textbf{Outline.} In Section \ref{sec:background} we review some results about the cohomology of the moduli space of stable rank two holomorphic bundles over a Riemann surface with fixed odd determinant. This background material is meant to motivate the statements of our results and the reader interested only in the proofs can likely skip this section. In Section \ref{sec:ring} we review the description of Mu\~{n}oz's ring. In Section \ref{sec:nilpotent} we prove Theorem \ref{prop:main}. In Section \ref{sec:framed} we prove Theorem \ref{thm:main}. Finally, we have included a short appendix at the end explaining how the exact sequence (\ref{eq:twistedgysin}) arises.\\
\vspace{.65cm}

\noindent \textbf{Acknowledgments.} The authors thank Kim Fr\o yshov for his encouragement and helpful comments. Thanks also to Aliakbar Daemi and Matt Stoffregen for helpful comments. The second author was supported by NSF grant DMS-1503100.\\

\vspace{.35cm}

\newpage

%%%%%%%%%%%%%%%%%%%

\section{The analogy with singular cohomology}\label{sec:background}

The results stated in the introduction can be compared to analogous results in ordinary singular cohomology. In place of Mu\~{n}oz's ring is the long-studied cohomogy of the moduli space of stable rank two holomorphic bundles with fixed odd determinant over a Riemann surface. To prove Theorems \ref{prop:main} and \ref{thm:main} we work almost entirely in Mu\~{n}oz's ring, and thus do not require any mention of this cohomology ring. However, in order to provide a proper narrative for the reader, we outline the relevant part of that story here. This ``analogy'' is just a comparison of the $E^1$- and $E^\infty$-pages of Fukaya's Bott-Morse type spectral sequence to which we alluded in the introduction.

\vspace{.75cm}

\noindent\textbf{{{The moduli space of stable rank two holomorphic bundles}}}\\

\noindent Let $N^g$ be the moduli space of stable rank two holomorphic bundles with fixed odd determinant over a Riemann surface $\Sigma$ of genus $g$. This space is a $6g-6$ dimensional closed and simply-connected symplectic manifold. For an introductory survey on $N^g$ we refer the reader to \cite{thaddeus-survey}. Here we discuss some of the basic properties that put our results in context.

Using a classical theorem of Narasimhan-Seshadri, $N^g$ may be identified with the moduli space of flat connections on a non-trivial $SO(3)$-bundle over $\Sigma$ modulo gauge transformations that lift to $SU(2)$. Via holonomy, $N^g$ is then topologically equivalent to $f_g^{-1}(-I)/SU(2)$, in which\\

\vspace{-.10cm}
\[
    f_g:SU(2)^{2g} \longrightarrow SU(2), \qquad f_g(A_1,B_1,\ldots,A_g,B_g) = \prod_{i=1}^{g} [A_i,B_i]
\]
\vspace{.25cm}

\noindent and the action of $SU(2)$ is by conjugation. Otherwise said, let $\Sigma_0$ be the surface with circle boundary resulting from deleting a small open disk from $\Sigma$. Then the above topological model for $N^g$ is that of conjugacy classes of homomorphisms $\pi_1(\Sigma_0)\longrightarrow SU(2)$ that send the class of an oriented loop traversing the boundary to $-I$. 

There is a universal rank two complex vector bundle over $N^g\times \Sigma$, and the restriction of its endomorphism bundle to $N^g\times \{\text{pt}\}$ has structure group $SO(3)=PU(2)$. We write\\

\begin{equation}
\begin{tikzcd}
    SO(3) \arrow[hookrightarrow]{r} & N^g_0 \arrow{r}{p} & N^g\label{eq:fibration}
\end{tikzcd}
\end{equation}
\vspace{.20cm}

\noindent for the associated principal $SO(3)$-bundle. The {\emph{framed}} moduli space $N_0^g$ is topologically equivalent to the $6g-3$ dimensional manifold $f_g^{-1}(-I)$, and under this equivalence the bundle projection $p$ coincides with the conjugation action projection. 

The betti numbers of the singular cohomology ring $H^\ast(N^g)$ were computed by Newstead \cite{newstead-topological}, Harder-Narasimhan \cite{harder-narasimhan} and Atiyah-Bott \cite{atiyah-bott}. Here and throughout the paper, we use complex coefficients for (co)homology. The result can be expressed via the Poincar\'{e} polynomial:\\

\begin{equation}
    P_t(N^g) = \frac{(1+t^3)^{2g} - t^{2g}(1+t)^{2g}}{(1-t^2)(1-t^4)},\qquad P_t(N^g) := \sum_{i=0}^{6g-6} {\dim H^i(N^g) t^i}.\label{eq:poincare}
\end{equation}
\vspace{.30cm}

\noindent For $g\geqslant 2$, note that the euler characteristic $\chi(N^g) = P_{-1}(N^g) =0$. Later, a simpler proof of (\ref{eq:poincare}) was given by Thaddeus, who used a perfect Bott-Morse function for $N^g$ \cite{thaddeus-survey}. 

The cohomology $H^\ast(N_0^g)$ is related to that of $N^g$ through a Gysin sequence. Following the convention of the literature, we let $\beta$ denote the first Pontryagin class of the $SO(3)$ fibration (\ref{eq:fibration}). In particular, $\beta\in H^4(N^g)$. Then the Gysin sequence is a long exact sequence\\

\begin{equation}
\begin{tikzcd}
   \cdots \arrow{r} &  H^{\ast-1}(N_0^g) \arrow{r} & H^{\ast-4}(N^g) \arrow{r}{\cup\beta} & H^{\ast}(N^g) \arrow{r}{p^\ast} & H^{\ast}(N_0^g) \arrow{r} & \cdots \label{eq:gysin}
\end{tikzcd}
\end{equation}

\vspace{.45cm}

\noindent Thus the betti numbers of $H^\ast(N_0^g)$ may be computed from the kernel of the endomorphism on $H^\ast(N^g)$ given by cup product with $\beta$. With the ring structure of $H^\ast(N^g)$ now very well understood, as we will mention in Section \ref{sec:ring}, this is a rather straightforward algebra problem. The betti numbers however were computed earlier by Newstead \cite{newstead-topological} and with $h^g_i :=\dim H^i(N^g_0)$ they are:\\

\begin{equation}
    h^g_i = \sum_{\substack{k=i-2g+2 \\ k\equiv i \text{(mod 2)}}}^{\lfloor i/3 \rfloor} {2g \choose k} \quad \text{for }i<3g-1,\qquad h_i^g=h^g_{6g-3-i}\quad\text{for }i\geqslant 3g-1.\label{eq:newsteadbetti}
\end{equation}
\vspace{.20cm}

\noindent We note that the integral cohomology of $N^g$ is always torsion-free, while the integral cohomology of $N^g_0$ generally contains 2-torsion. For example, $N^1$ is a point, and $N^1_0$ is a copy of $SO(3)$. We remark that the betti number $n_i$ of Table \ref{fig:table2} in the introduction is equal to twice the sum of the $h_j$ over all $j$ congruent to $i$ (mod 4). From the formulae (\ref{eq:newsteadbetti}) one can compute that the total rank of $H^\ast(N^g_0)$ is equal to $g{2g \choose g}$ (see also Section \ref{sec:comparison}).

Related to the kernel of $\beta$ is the fact that multiplication by $\beta$ is nilpotent. In \cite{newstead-char} Newstead conjectured that $\beta^g=0$ in the ring $H^\ast(N^g)$, and observed that this relation is equivalent to the vanishing of all Pontryagin numbers of $N^g$. The slightly stronger statement\\

\vspace{-.10cm}
\begin{equation}\label{eq:betanilp}
    \text{min}\Big\{n\geqslant 1: \;\; \beta^n = 0  \in H^\ast(N^g) \Big\} \;=\; g
\end{equation}
\vspace{.05cm}

\noindent was implicit in the work of Thaddeus, who computed the intersection pairings in the ring $H^\ast(N^g)$. He mentions in \cite[\S 5]{thaddeus-conformal} that $\beta^g=0$, and from his intersection pairing formula for $\alpha^n\beta^m$ one can read off that $\alpha^{g-1}\beta^{g-1}[N^g] = (-1)^g2^{2g-2}(g-1)!$, which implies $\beta^{g-1}\neq 0$.\\

\vspace{.60cm}

\noindent \textbf{The instanton homology of a surface times a circle}\\

\noindent 
Now we return to instanton homology. We first provide some background, and then relate our results in the introduction to the results about $H^\ast(N^g)$ listed above.

For a closed, oriented, connected 3-manifold $Y$ and a Hermitian line bundle $w\longrightarrow Y$ such that $c_1(w)$ has odd pairing with some integral homology class, following \cite{km-sutures} we denote Floer's relatively $\Z/8$-graded instanton homology by $I(Y)_w$. As in the introduction, we call the pair $(Y,w)$ {\emph{non-trivial admissible}}. Let $E$ be a rank two Hermitian vector bundle over $Y$ with determinant line bundle $w$. The chain complex for $I(Y)_w$ is generated by (irreducible, perturbed) projectively flat connections on $E$, and the differential counts (perturbed) instantons on $E \times \R$. An instanton is a connection $A$ on a bundle over a 4-manifold whose curvature $F_A$ satisfies the ASD equation

\vspace{.06cm}
\begin{equation*}
    F_A \; + \; \star F_A \; = \; 0
\end{equation*}
\vspace{.00cm}

\noindent where $\star$ is the Hodge star operation. The isomorphism class of the complex vector space $I(Y)_w$ depends only on $Y$ and the isomorphism class of the adjoint bundle associated to $E$ with structure group $SO(3)$, which is determined by $w_2(E)$.

The gauge transformations used in the construction of $I(Y)_w$ are those automorphisms of $E$ that have fiberwise determinant equal to one. Let $R$ be an oriented surface in $Y$ that has odd pairing with $c_1(w)$ and let $\xi_R$ be a real line bundle with $w_1(\xi_R)$ dual to $R$. Then $E \longmapsto E\otimes \xi_R$ gives rise to a map on the space of connections. The induced effect on $I(Y)_w$ is a degree four involution. We write $I(Y)_{w}'$ for the relatively $\Z/4$-graded vector space obtained by modding out this involution. Alternatively, this is the group obtained by using a slightly larger gauge group that contains the determinant one transformations as an index two subgroup.

Let $w\longrightarrow \Sigma\times S^1$ have first chern class dual to the circle factor. The critical set of the Chern-Simons functional mod determinant one gauge transformations is identified with

\vspace{.15cm}

\[
    \Big\{ \rho\in \text{Hom}(\pi_1(\Sigma_0\times S^1),SU(2)): \;\; \rho(\partial \Sigma_0\times \text{pt.})=-I\Big\}/SU(2).
\]
\vspace{.00cm}

\noindent Each such representation $\rho$ evaluates as $+I$ or $-I$ around the circle factor, and this splits the set into two identical pieces, each a copy of $N^g$. As mentioned in the introduction, if certain transversality assumptions are met, Fukaya \cite{fukaya} shows that there is a spectral sequence with $E^1$-page two copies of $H^\ast(N^g)$ that converges to $I(\Sigma\times S^1)_w$. We mention here that we ignore the distinction between homology and cohomology in this context, identifying them via a duality isomorphism. The spectral sequence in this situation collapses and we have:

\vspace{.10cm}
\begin{equation}
    I(\Sigma\times S^1)_{w}' \;\cong\; H^\ast(N^g).\label{eq:veciso}
\end{equation}
\vspace{-.02cm}

\noindent The prime superscript on the left side of (\ref{eq:veciso}) has the effect of identifying the two copies of $H^\ast(N^g)$. In actuality, (\ref{eq:veciso}) is proven \cite{munoz-ring} without mention of spectral sequences, and the above collapsing is a corollary. We mention that (\ref{eq:veciso}) also follows from work of Dostoglou-Salamon \cite{ds}. The isomorphism is $\Z/4$-graded, if one collapes the $\Z$-grading on $H^\ast(N^g)$ modulo 4. Thus the Poincar\'{e} polynomial (\ref{eq:poincare}) determines the betti numbers of $I(\Sigma\times S^1)_{w}$.

We emphasize that (\ref{eq:veciso}) is an isomorphism of {\emph{vector spaces}}. There is a ring structure on $I(\Sigma\times S^1)_{w}'$ that we will review in Section \ref{sec:ring}, but (\ref{eq:veciso}) is {\emph{not}} an isomorphism of rings: the instanton homology ring has a product which is a deformation of the cup product in $H^\ast(N^g)$. Also, we mention:\\

\vspace{.30cm}

\noindent \textbf{Notation:} {\emph{The $u$-map on $I(\Sigma\times S^1)'_w$ will be denoted $\beta$, in alignment with Mu\~{n}oz's notation.}}\\

\vspace{.30cm}

\noindent Now we turn to the results stated in the introduction. Theorem \ref{prop:main} regarding the nilpotency degree of $\beta^2-64$ in the ring $I(\Sigma\times S^1)_w'$ is the instanton analogue of (\ref{eq:betanilp}), the nilpotency of $\beta$ in the ring $H^\ast(N^g)$. Note that the powers of $\beta$ appearing in the minimal nilpotency relations in the two separate cases differ roughly by a factor of two.

Next, the Gysin sequence (\ref{eq:gysin}) for the fibration (\ref{eq:fibration}) may be viewed as the singular cohomology analogue of the instanton exact sequence (\ref{eq:twistedgysin}) in the introduction as applied to a surface times a circle. The former exact sequence represents the framed cohomology $H^\ast(N_0^g)$ as the mapping cone of multiplication by $\beta$ on the ring $H^\ast(N^g)$, while the latter represents $I^\#(\Sigma\times S^1)_w$ as the mapping cone of multiplication by $\beta^2-64$ on the ring $I(\Sigma\times S^1)_w'$. 

Finally, the comparison of the betti numbers in Tables \ref{fig:table1} and \ref{fig:table2} shows that in contrast to the isomorphism (\ref{eq:veciso}) in the unframed situation, the framed instanton homology $I^\#(\Sigma\times S^1)_w$ is {\emph{not}} simply the homology of the critical set $N_0^g\sqcup N_0^g$. Assuming that the construction of the spectral sequence (\ref{eq:specseq}) carries through, one may think of this non-collapsing as (roughly) an indication that there exist isolated instantons (modulo translation, and for generic perturbations) on the relevant non-trivial bundle over the cylinder $(\Sigma\times S^1\# T^3)\times \R$. 

%It would be interesting to understand the non-trivial differentials of this spectral sequence, and its construction in general, from a purely algebraic perspective.

\vspace{.35cm}

%%%%%%%%%%%%%%%
\section{The ring structure of the instanton homology} \label{sec:ring}

We review the ring structure of the instanton Floer homology $I(\Sigma\times S^1)'_w$ given by Mu\~{n}oz. The ring multiplication on this vector space is defined as follows: a pair of pants cobordism from $S^1\sqcup S^1$ to $S^1$ crossed with $\Sigma$ yields a (3+1)-dimensional cobordism, inducing the product map

\vspace{.10cm}

\[
    I(\Sigma\times S^1)_w'\otimes I(\Sigma\times S^1)_w' \longrightarrow I(\Sigma\times S^1)'_w.
\]
\vspace{.05cm}

\noindent The relative Donaldson invariants from $X=\Sigma\times D^2$, with the bundle data $w$ over the boundary $\Sigma\times S^1$ extended in an appropriate way, generate this ring. More precisely, if $\phi^w(X,z)$ denotes the relative Donaldson invariant in $I(\Sigma\times S^1)_w'$ for some $z\in \mathbb{A}(X):=\text{Sym}^\ast\left(H_0(X)\oplus H_2(X)\right) \otimes \Lambda^\ast H_1(X)$, the following are generators, where $\gamma_i$ runs over a $2g$ element basis of $H^1(\Sigma)\subset \mathbb{A}(X)$:

\vspace{.10cm}
\[
    \alpha = 2\phi^w(X,\Sigma), \qquad \psi_i = \phi^w(X,\gamma_i), \qquad \beta=-4\phi^w(X,x).
\]
\vspace{.10cm}

\noindent The endomorphism $u$ from the introduction is multiplication by $\beta$. Next, the action of the diffeomorphism group of $\Sigma$ on Floer homology factors through an action of $\text{Sp}(2g,\Z)$. One can decompose the ring into summands using this action. To facilitate this, let $H=H^1(\Sigma)$, and suppose $\{\gamma_i,\gamma_{i+g}\}_{i=1}^g$ is a symplectic basis for $H$ with $\gamma_i\cdot \gamma_{i+g} = 1$. We define the primitive component $\Lambda^k_0 H$ of $\Lambda^k H$:
\vspace{.10cm}

\[
    \Lambda^k_0 H = \text{ker}\left(\;\;\cdot \wedge \gamma^{g-k+1}:\;\Lambda^k H \longrightarrow \Lambda^{2g-k+2} H\right), \qquad \gamma := -2\sum_{i=1}^g \gamma_i\wedge \gamma_{i+g}.
\]
\vspace{.10cm}

\noindent Mu\~{n}oz gave the following explicit description of the instanton Floer homology ring:\\
\vspace{.10cm}

\begin{theorem}[\cite{munoz-ring}]\label{thm:munoz}
    The ring $I(\Sigma\times S^1)_w'$ is isomorphic to
\begin{equation}
    \bigoplus_{k=0}^g \Lambda_0^k H \otimes \C[\alpha,\beta,\gamma]/J_{g-k}\label{eq:ring}
\end{equation}
where $J_{k}=(\zeta_k,\zeta_{k+1},\zeta_{k+2})$ are ideals with generators inductively defined by $\zeta_i=0$ for $i<0$, $\zeta_0=1$ and\\

\begin{equation}
    \zeta_{k+1} = \alpha\zeta_k + k^2(\beta + (-1)^k8)\zeta_{k-1} + 2k(k-1)\gamma\zeta_{k-2}.\label{eq:rels}
\end{equation}
\end{theorem}
\vspace{.65cm}

\noindent Under this isomorphism, $\alpha$ and $\beta$ are as defined above, while $\psi_i$ corresponds to $\gamma_i$.

If the term $(-1)^k8$ in the recurrence relation (\ref{eq:rels}) is {\emph{replaced by zero}} for each $k$, then the result is instead the cohomology ring $H^\ast(N^g)$. Such a presentation for $H^\ast(N^g)$ was at least in part conjectured by Mumford and established by several authors \cite{king-newstead, st, b,zagier}. A complete set of relations for the cohomology ring $H^\ast(N^g)$ was found earlier by Kirwan \cite{kirwan}.

Thus the instanton Floer homology is a deformation of the ring $H^\ast(N^g)$. We mention two other well-known rings that are isomorphic to this instanton homology, both of which rely on $N^g$ carrying a natural symplectic structure. First, there is the {\emph{quantum cohomology}} of $N^g$. This isomorphism was established by Mu\~{n}oz in \cite{munoz-quantum}. Second, there is the {\emph{symplectic Floer homology}} of $N^g$, with respect to the identity symplectomorphism. This latter ring was earlier known to be isomorphic to the aforementioned quantum cohomology ring \cite{pss}.

\vspace{.5cm}

%%%%%%%%%%%%%%%
\section{Computing the nilpotency degree} \label{sec:nilpotent}

In this section we prove Theorem \ref{prop:main}. As before, $g\geqslant 1$ denotes the genus of our surface $\Sigma$. Using the ring isomorphism (\ref{eq:ring}), and identifying $u$ with multiplication by $\beta$, the statement of Proposition \ref{prop:main} transforms into one about the ideals $J_g \subset \C[\alpha,\beta,\gamma]$. Two easily verified properties of this sequence of recursively defined ideals are the inclusions $J_g \subset J_{g-1}$ and $\gamma J_g \subset J_{g+1}$. Now, if we define

\vspace{.10cm}
\[
    n_g \;\; = \;\; \text{min}\left\{n\geqslant 1: \;\; (\beta^2-64)^n \in J_g \right\}
\]
\vspace{.03cm}

\noindent then Theorem \ref{prop:main} is equivalent to $n_g = 2\lceil g/2 \rceil -1$. For convenience, we define:
\[
    \beta_r = \beta + (-1)^r8, \qquad \beta_{+} = \beta + 8, \qquad \beta_{-} = \beta - 8.
\]
Next, we introduce notation for certain monomials in $\beta_\pm$ that help formalize the structure of the proof. First, define $\phi_r = \beta_-^{\lfloor r/2 \rfloor +1}\beta_+^{\lceil r/2 \rceil}$ and $\psi_r = \beta_-^{\lfloor r/2 \rfloor }\beta_+^{\lceil r/2 \rceil}$. The basic properties of these are:

\vspace{.10cm}
\begin{itemize}
\item[]
\begin{enumerate}
\item $\phi_0=\beta_-$ and $\psi_0=1$.
\item For all $r$, $\phi_{r+1}=\beta_r\phi_r$ and $\psi_{r+1}=\beta_r\psi_r$.
\item For all $r$, $\phi_r = \beta_-\psi_r$.
\end{enumerate}
\end{itemize}

\vspace{.2cm}

\noindent We also define $\rho_j = \beta_-^{2\lfloor \frac{j-1}{2}\rfloor}\beta_+^{j-1}$ and $\eta_j=\beta_-^{j-1}\beta_+^{j-1}$ for $j \geqslant 1$. For technical reasons we also define $\rho_j=1$ for $j<1$. We have the basic properties:

\vspace{.10cm}
\begin{itemize}
\item[]
\begin{enumerate}
\item $\rho_1=\eta_1=1$.
\item If $j>1$ is odd, then $\rho_j=\beta_-^2\beta_+\rho_{j-1}$ and $\rho_j=\eta_j$.
\item If $j$ is even, then $\rho_j=\beta_+\rho_{j-1}$ and $\beta_-\rho_j=\eta_j$.
\end{enumerate}
\end{itemize}

\vspace{.2cm}

\begin{lemma} \label{lemma:above}
For all natural numbers $r\geqslant 1$ and $0\leqslant j \leqslant r$, we have:
\begin{itemize}
\item[]
\begin{enumerate}
\item[{\emph{(i)}}] $\rho_j\phi_{r-j}\zeta_{r-j} \in J_r$.
\item[{\emph{(ii)}}] $\alpha\eta_j\psi_{r-j}\zeta_{r-j}\in J_r$.
\end{enumerate}
\end{itemize}
\end{lemma}

\vspace{.2cm}
\begin{proof}
The proof is an induction on $(r,j)$ under the lexicographic ordering. Note that the cases $(r,0)$ follow from the definition of $J_r$. When $j>r$ the statement of the lemma is true if we define $\zeta_{k}=0$ for $k<0$. Suppose the statement holds below $r$ and up to but not including $j$ at $r$. In the recursive relation (\ref{eq:rels}), set $k=r-j+1$, multiply both sides by $\rho_j\phi_{r-j-1}$, and rearrange, to obtain:
\vspace{.15cm}
\begin{equation}
    \rho_j\phi_{r-j}\zeta_{r-j} = c_1\underbracket[.01cm]{\rho_j\phi_{r-j-1}\zeta_{r-j+2}}_{\text{(I)}} + c_2 \underbracket[.01cm]{\alpha \rho_j\phi_{r-j-1}\zeta_{r-j+1}}_{\text{(II)}} + c_3 \underbracket[.01cm]{\gamma \rho_j\phi_{r-j-1}\zeta_{r-j-1}}_{\text{(III)}}\label{eq:lemma}
\end{equation}
\vspace{.15cm}

\noindent where the $c_i$ are rational numbers. We have labelled the terms on the right side (I), (II), (III) from left to right. We show that these three terms are in $J_r$ using our induction hypothesis. We begin with term (I). We compute that the factor in front of $\zeta_{r-j+2}$ in term (I) is
\[
    \rho_j\phi_{r-j-1}=\beta_-^2\beta_+^2\rho_{j-2}\phi_{r-j-1} =\beta_{r-j}\rho_{j-2}\phi_{r-j+2},
\]
implying term (I) is in $J_r$ by the induction hypothesis at $(r,j-2)$. Term (III) is in $\gamma J_{r-1}\subset J_r$ by the induction hypothesis at $(r-1,j)$. Finally, we consider term (II). If $j$ is even, then 
\[
    \alpha\rho_j \phi_{r-j-1}= \alpha\beta_-\beta_+\rho_{j-1}\psi_{r-j-1}=\alpha\rho_{j-1}\psi_{r-j+1} = \alpha\eta_{j-1}\psi_{r-j+1}.
\]
If $j$ is odd, we compute the same, but times $\beta_-$. Thus term (II) is in $J_r$ using the induction hypothesis at $(r,j-1)$. We conclude $\rho_j\phi_{r-j}\zeta_{r-j}\in J_r$, proving (i) at the induction step $(r,j)$.

Now we prove (ii) at step $(r,j)$. The argument has a similar structure. With $k=r-j+1$, multiply both sides of (\ref{eq:rels}) by $\alpha\eta_j\psi_{r-j-1}$ and rearrange, to obtain:
\vspace{.15cm}
\[
    \alpha \eta_j \psi_{r-j} \zeta_{r-j} = c_1\underbracket[.01cm]{\alpha\eta_j\psi_{r-j-1}\zeta_{r-j+2}}_{\text{(I)}} + c_2 \underbracket[.01cm]{\alpha^2 \eta_j\psi_{r-j-1}\zeta_{r-j+1}}_{\text{(II)}} + c_3 \underbracket[.01cm]{\alpha\gamma \eta_j\psi_{r-j-1}\zeta_{r-j-1}}_{\text{(III)}}.
\]
\vspace{.15cm}

\noindent Now $\alpha\eta_j\psi_{r-j-1}$ is equal to $\alpha \beta_{r-j}\eta_{j-2}\psi_{r-j+2}$, whence term (I) is in $J_{r}$ by the induction hypothesis at $(r,j-2)$. Next, $\alpha\eta_j\psi_{r-j-1}$ is equal to $\alpha\eta_{j-1}\psi_{r-j+1}$, so term (II) is also in $J_r$ by induction at $(r,j-1)$. Finally, the third term is in $\gamma J_{r-1}\subset J_r$ by the inductive assumption at $(r-1,j)$. Therefore the left-hand term $\alpha\eta_j\psi_{r-j}\zeta_{r-j}\in J_r$, completing the proof of (i) and (ii) by induction.
\end{proof}
\vspace{.6cm}

\noindent We will soon see that part (i) of Lemma \ref{lemma:above} provides the desired upper bound for $n_g$. Part (ii) is only included as an extra inductive assumption in order to carry through the proof of (i). We also need to bound $n_g$ from below, for which we use a lemma regarding non-inclusion. For this we use the following fact from \cite{munoz-ring}: the images of monomials

\vspace{.07cm}
\[
    \alpha^i\beta^j\gamma^k, \qquad i+j+k < g,\quad i,j,k\geqslant 0
\]
\vspace{.03cm}

\noindent form a vector space basis for the quotient ring $\C[\alpha,\beta,\gamma]/J_g$, which is the $\text{Sp}(2g,\Z)$-invariant part of the ring $I(\Sigma\times S^1)_w'$. In particular, we have the non-inclusion $\gamma^{g-1}\not\in J_g$.\\

\vspace{.20cm}

\begin{lemma} \label{lemma:below}
Suppose $g$ is odd and $j\leqslant (g-1)/2$. Then $\beta_-^{j+(g-1)/2}\beta_+^{g-1}\zeta_{g-2j-1}$ is equivalent modulo the ideal $J_g$ to a positive rational multiple of $\gamma^{g-1}$.
\end{lemma}

\vspace{.2cm}

\begin{proof}

We use the notation $a \equiv_g b$ to mean that $a$ and $b$ are equivalent modulo $J_g$. The proof is by induction: first on $g$, then on $j$. The base case $g=1$ is trivial. Now we handle the induction step for $g>1$; since $g$ is odd, in fact $g\geqslant 3$. If $j=0$ then we compute:

\begin{align*}
    \beta_-^{(g-1)/2}\beta_+^{g-1}\zeta_{g-1} & \;=\;\;\; g^{-2}\beta_-^{(g-3)/2}\beta_+^{g-1}\left(\zeta_{g+1} - \alpha\zeta_{g}-2g(g-1)\gamma\zeta_{g-2}\right)\\
    & \;\equiv_g \; -2g^{-1}(g-1)\beta_-^{(g-3)/2}\beta_+^{g-1}\gamma\zeta_{g-2}\\
    & \; = \;\;\;  -2g^{-1}(g-1)^{-1}\beta_-^{(g-3)/2}\beta_+^{g-2}\gamma\left(\zeta_g -\alpha\zeta_{g-1}-2(g-1)(g-2)\gamma\zeta_{g-3}\right)\\
    & \;\equiv_g \; 4g^{-1}(g-2)\beta_-^{(g-3)/2}\beta_+^{g-2}\gamma^2\zeta_{g-3}\\
    & \; = \;\;\; 4g^{-1}(g-2)\beta_-^{(g-3)/2}(\beta_-+16)\beta_+^{g-3}\gamma^2\zeta_{g-3}\\
    & \; = \;\;\; 4g^{-1}(g-2)\beta_-^{(g-1)/2}\beta_+^{g-3}\gamma^2\zeta_{g-3} + 16\cdot 4g^{-1}(g-2)\beta_-^{(g-3)/2}\beta_+^{g-3}\gamma^2\zeta_{g-3}\\
    & \; \equiv_g \; 64g^{-1}(g-2)\beta_-^{(g-3)/2}\beta_+^{g-3}\gamma^2\zeta_{g-3}.
\end{align*}

\noindent The last equivalence follows because Lemma \ref{lemma:above} implies that $\beta_-^{(g-1)/2}\beta_+^{g-3}\zeta_{g-3}\in J_{g-2}$ and therefore the first term $4g^{-1}(g-2)\beta_-^{(g-1)/2}\beta_+^{g-3}\gamma^2\zeta_{g-3}$ of the penultimate line is a member of $J_g$. By the induction hypothesis at $(g-2,0)$, we conclude that $\beta_-^{(g-1)/2}\beta_+^{g-1}\zeta_{g-1}$ is a positive rational multiple of $\gamma^{g-1}$ modulo $J_g$. If $j>0$ then begin similarly as above to obtain

\begin{align*}
    \beta_-^{j+(g-1)/2}\beta_+^{g-1}\zeta_{g-2j-1} \; = \; & \;\;\;\, (g-2j)^{-2}\beta_-^{j+(g-3)/2}\beta_+^{g-1}\zeta_{g-2j+1}\\
    & -(g-2j)^{-2}\beta_-^{j+(g-3)/2}\beta_+^{g-1}\alpha\zeta_{g-2j}\\
    &- 2(g-2j-1)(g-2j)^{-1}\beta_-^{j+(g-3)/2}\beta_+^{g-1}\gamma\zeta_{g-2j-2}.
\end{align*}
\vspace{.10cm}

\noindent By the inductive assumption at $(g,j-1)$, the first term on the right hand side is a positive rational multiple of $\gamma^{g-1}$ modulo $J_g$. By Lemma \ref{lemma:above} (ii), the second term on the right is a member of $J_g$. It remains to show that the third term on the right is a positive rational multiple of $\gamma^{g-1}$ modulo $J_g$. Using the recursive definition of $\zeta_{g-2j}$, we have

\begin{align*}
    -\beta_-^{j+(g-3)/2}\beta_+^{g-1}\gamma\zeta_{g-2j-2} \; = \; & -(g-2j-1)^{-2}\beta_-^{j+(g-3)/2}\beta_+^{g-2}\gamma\zeta_{g-2j}\\
    & +(g-2j-1)^{-2}\beta_-^{j+(g-3)/2}\beta_+^{g-2}\alpha\gamma\zeta_{g-2j-1}\\
    & +2(g-2j-2)(g-2j-1)^{-1}\beta_-^{j+(g-3)/2}\beta_+^{g-2}\gamma^2\zeta_{g-2j-3}.
\end{align*}
\vspace{.10cm}

\noindent By Lemma \ref{lemma:above}, the first and second terms on the right hand side are members of $J_{g-1}$. The third term on the right side is a positive rational multiple of $\gamma^{g-1}$ modulo $J_g$ because

\vspace{.05cm}
\[
    \beta_-^{j+(g-3)/2}\beta_+^{g-2}\gamma^2\zeta_{g-2j-3} \; \equiv_g \; 16\beta_-^{j+(g-3)/2}\beta_+^{g-3}\gamma^2\zeta_{g-2j-3}
\]
\vspace{.10cm}

\noindent is a positive rational multiple of $\gamma^{g-1}$ modulo $J_{g}$ by the inductive assumption at $(g-2,j)$, for which the inclusion $\gamma^2 J_{g-2} \subset J_g$ has been used.
\end{proof}

\vspace{.65cm}

\noindent We now deduce Theorem \ref{prop:main} from these two lemmas.\\

\vspace{.25cm}

\begin{proof}[Proof of Theorem \ref{prop:main}]
We first use Lemma \ref{lemma:above} to prove $n_g\leqslant 2\lceil g/2 \rceil -1$. By (i) of that lemma,

\vspace{.05cm}
\begin{equation}
    \rho_g\phi_0 = \beta_-^{2\lfloor \frac{g-1}{2} \rfloor + 1}\beta_+^{g-1} \in J_g.\label{eq:uselemma}
\end{equation}
\vspace{.10cm}

\noindent If $g$ is even, this immediately implies the inequality, while if $g$ is odd, we multiply (\ref{eq:uselemma}) by $\beta_+$. To obtain the reverse inequality, we first consider Lemma \ref{lemma:below}, which for $g$ odd implies

\vspace{.05cm}
\begin{equation*}
     \beta_-^{g-1}\beta_+^{g-1} \; \equiv_g\; c\cdot \gamma^{g-1}\not\in J_g\label{eq:uselemma2}
\end{equation*}
\vspace{.10cm}

\noindent where $c>0$. We conclude that $n_g\geqslant 2\lceil g/2 \rceil -1$ for $g$ odd. Next, since $J_g\subset J_{g-1}$, we have $n_{g} \geqslant n_{g-1}$, proving the same inequality for $g$ even, and completing the proof of Theorem \ref{prop:main}.
\end{proof}

\vspace{.55cm}

%%%%%%%%%%%%%%%
\section{The framed instanton homology} \label{sec:framed}

\vspace{.25cm}

In this section and the contained subsections we prove Theorem \ref{thm:main}. We first briefly discuss some notation. For a $\Z/4$-graded vector space $V$ we define its Poincar\'{e} polynomial to be

\vspace{.18cm}
\[
    P_t(V)\;= \; \dim V_0\cdot t^0 + \dim V_1\cdot t^1  + \dim V_2 \cdot t^2 + \dim V_3 \cdot t^3
\]
\vspace{.05cm} 

\noindent in which $V_i$ is the grading $i$ (mod 4) summand of $V$. We think of $P_t(V)$ as an element of the ring $\Z[t]/(t^4-1)$. The Poincar\'{e} polynomial of a tensor product (resp. direct sum) of $\Z/4$-graded vector spaces is the product (resp. sum) of the Poincar\'{e} polynomials of the factors. When $V$ is a quotient ring of the form $R/J$ where $J$ is an ideal, we write $P_t(J)$ to mean $P_t(V)$.

The proof of Theorem \ref{thm:main} amounts to a computation of the $\Z/4$-graded Poincar\'{e} polynomial of the framed instanton homology. This may be expanded as

\vspace{.25cm}
\begin{equation}
    P_t\left(I^\#(S^1\times \Sigma)_w\right) \;= \; (1+t^3)\sum_{k=0}^g \left({2g \choose k} - {2g \choose k-2} \right)t^{3k}P_t\left(K_{g-k}\right)\label{eq:poincarepolynomial}
\end{equation}
\vspace{.18cm}

\noindent in which $K_{g}$ is defined to be the kernel of multiplication by $\beta^2-64$ on the ring $\C[\alpha,\beta,\gamma]/J_g$. This expression is explained as follows. First, the general shape comes from the $\text{Sp}(2g,\Z)$-decomposition of $I(\Sigma\times S^1)_w'$ from Theorem \ref{thm:munoz}, where we have used that $\dim \Lambda^k_0 H = {2g\choose k} - {2g \choose k-2}$. Next, the terms $(1+t^3)P_t(K_{g-k})$ appear because of the mapping cone exact sequence (\ref{eq:twistedgysin}), characterizing the framed homology as the kernel plus cokernel of $u^2-64$. The factors $t^{3k}$ are included because the homogeneously graded vector space $\Lambda_0^k H$ has grading $3k$.

Save for some elementary manipulations, this reduces the computation to that of $P_t(K_g)$. As suggested in the introduction, we will break this into two pieces. We write $K_g^\pm$ for the kernel of multiplication by $\beta\pm 8$ on the ring $\C[\alpha,\beta,\gamma]/J_g$. The nilpotency of $\beta^2-64$ readily implies that

\vspace{.20cm}
\[
    P_t(K_g) \; = \; P_t(K_g^+) \; + \; P_t(K_g^-).
\]
\vspace{.10cm}

\noindent As $K_g^\pm$ is by definition the subspace of the ring $\C[\alpha,\beta,\gamma]/J_g$ on which $\beta$ acts as $\mp 8$, we may describe it as follows. Define $J^{\pm}_g$ to be the ideal in $\C[\alpha,\gamma]$ which is the image of $J_g$ under the evaluation map setting $\beta= \pm 8$. The definition of $J_g^{\pm}$ has the same recursion shape as before:

\vspace{.10cm}
\[
    J_g^{\pm} = (\zeta_g^\pm, \zeta_{g+1}^\pm, \zeta_{g+2}^\pm)
\]
\vspace{.05cm}

\noindent in which $\zeta_k^\pm = 0$ for $k<0$, the starting term is $\zeta_0^\pm =1$, and

\vspace{.05cm}
\begin{align}
    k \text{ even}:\qquad &\zeta^+_{k+1} = \alpha \zeta_k^+ + 16 k^2 \zeta_{k-1}^+ + 2k(k-1)\gamma\zeta_{k-2}^+ \label{eq:pe}\\
    k \text{  odd}:\qquad &\zeta^-_{k+1} = \alpha \zeta_k^-  - 16 k^2 \zeta_{k-1}^- + 2k(k-1)\gamma\zeta_{k-2}^-\label{eq:mo}\\
    k \text{ even}:\qquad &\zeta^-_{k+1} = \alpha \zeta_k^- + 2k(k-1)\gamma\zeta_{k-2}^-\phantom{16 k^2 \zeta_{k-1}^+}\label{eq:me}\\
    k \text{  odd}:\qquad &\zeta^+_{k+1} = \alpha \zeta_k^+ + 2k(k-1)\gamma\zeta_{k-2}^+ \phantom{16 k^2 \zeta_{k-1}^+} \label{eq:po}
\end{align}
\vspace{.10cm}

\noindent Then $K_g^\pm$ is isomorphic to the quotient $\C[\alpha,\gamma]/J_g^\mp$. By definition we thus have $P_t(K_g^\pm) = P_t(J_g^\mp)$. We henceforth focus on computing the four coefficients of $P_t(J_g^\pm)$. We remind the reader that the $\Z/4$-gradings of the elements $\alpha$ and $\gamma$ are both $2$ (mod 4). Also, from the defining formulae (\ref{eq:pe})-(\ref{eq:po}) for the polynomials $\zeta_g^\pm$ we record the following:

\vspace{.30cm}

\begin{equation}
    \zeta^\pm_g \;=\; \alpha^g \;\;+\;\; \text{lower order terms}.\label{eq:degalpha}
\end{equation}
\vspace{.10cm}

\noindent We proceed to compute $P_t(J_g^-)$ and $P_t(J_g^+)$ separately. These are the Poincar\'{e} polynomials for the $\Z/4$-graded kernels of $\beta+8$ and $\beta-8$, respectively.\\

\vspace{.45cm}

\subsection{The kernel of $\beta+8$.}
%\noindent\textbf{Computing the kernel of $\beta+8$.}\\

\vspace{.25cm}

\noindent We first compute $P_t(J^-_g)$. The key, as usual, is an understanding of how $J_g^-$ is related to $J_r^-$ with $r<g$. The situation is rather immediate for even indices.\\

\vspace{.20cm}

\begin{lemma}\label{lemma:minuseven}
    For $g\geqslant 2$ even, $J_g^-$ is generated by $\zeta_{g}^-$ and $\gamma J^-_{g-2}$. Further, for $i \geqslant 0$, $\gamma^{i}\zeta_{g-2i}^-\in J_g^-$.
\end{lemma}
\vspace{.15cm}

\begin{proof}
For $g$ even, the generator $\zeta_{g+1}^-\in J^-_g$ can be replaced by $\gamma\zeta_{g-2}^-$, which is a non-zero rational multiple of $\zeta_{g+1}^--\alpha\zeta_{g}^-$ by (\ref{eq:me}). Similarly, the generator $\zeta_{g+2}^-$ can be replaced by $\gamma\zeta_{g-1}^-$ using the relation (\ref{eq:mo}), in which $k=g+1$. This implies that $J_g^-=(\zeta^-_g,\;\gamma J^-_{g-2})$, which is the first statement. The second statement follows inductively from the first.
\end{proof}

\vspace{.45cm}

\noindent This gives a recursion relation for the total dimension of $\C[\alpha,\gamma]/J_g^-$ when $g$ is even, in the following way. First, following standard terminology, when $J$ is an ideal in an algebra $R$ with $R/J$ of finite dimension, we define the {\emph{degree}} of $J$, written $\deg J$, to be $\dim R/J$. In particular, $\deg J_g^-$ is the sum of the four coefficients of $P_t(J_g^-)$. Next, consider the following standard exact sequence\\

\begin{equation*}
\begin{tikzcd}[column sep=large]
   0 \arrow{r} & \text{ker}(\gamma) \arrow{r} & \C[\alpha,\gamma]/J_{g}^- \arrow{r}{\gamma} &  \C[\alpha,\gamma]/J_{g}^- \arrow{r} & \text{coker}(\gamma) \arrow{r} & 0 
\end{tikzcd}
\end{equation*}
\vspace{.25cm}

\noindent in which the map $\gamma$ is induced from multiplication by $\gamma$. The above lemma together with (\ref{eq:degalpha}) implies that for $g\geqslant 2$ even, $\text{ker}(\gamma)=J_{g-2}^-/J_g^-$. The cokernel may be identified with $\C[\alpha]/I_g^-$ in which $I_g^-$ is defined as was $J_g^-$ by setting $\gamma=0$. Evidently, $I_g^-$ is a principal ideal generated by a degree $g$ polynomial in $\alpha$, and thus $\deg I_g^- =g$. From the above exact sequence we obtain:\\

\begin{equation}
    \deg J_g^-  \; - \; \deg J_{g-2}^-  \; = \; g, \qquad g\geqslant 2 \text{ and even}.\label{eq:minusevendegree}
\end{equation}
\vspace{.10cm}

\noindent We also have $\deg J_0^-=0$. Now we turn to the case in which $g$ is odd. Computations for low values of $g$ suggest that $J_g^-=J_{g-1}^-$ so that the odd case is covered by the even case. This amounts to showing $\zeta_{g-1}^-\in J_g^-$. Proving this relation, however, is not quite as straightforward as was Lemma \ref{lemma:minuseven}. Our method is similar to the proof of Lemma \ref{lemma:above} in which we prove two statements that help carry through an induction scheme. Here, however, we will need to take inverses of certain elements modulo $J_g^-$ for the argument to work. We thus first describe a way to test the invertibility of elements mod $J_g^-$. For this we use the following observation of Mu\~{n}oz:

\vspace{.45cm}

\begin{lemma}[{\cite[Prop. 20]{munoz-ring}}]\label{lemma:munoz}
    The simultaneous triples of eigenvalues of the commuting endomorphisms of multiplication by $\alpha, \beta,\gamma$ on the ring $\C[\alpha,\beta,\gamma]/J_g$ are given by\\
    
    \[
        (0,8,0),\;\; (\pm 4, -8, 0),\;\; (\pm 8\sqrt{-1}, 8,0), \;\; \ldots, \;\; (\pm 4(g-1)\sqrt{-1}^{g}, (-1)^{g-1} 8, 0). 
    \]
\end{lemma}
\vspace{.65cm}

\noindent In particular, as $\C[\alpha,\gamma]/J_g^-$ is identified with the subspace of $\C[\alpha,\beta,\gamma]/J_g$ for which $\beta$ acts as $-8$, multiplication by $\alpha$ on the former ring has only non-zero eigenvalues, and is thus invertible. To say something more general, suppose $g$ is odd, and for $j=1,\ldots, (g-1)/2$ consider the evaluation maps:

\vspace{.25cm}
\begin{equation*}
\begin{tikzcd}[column sep=large]
    \text{ev}^\pm_{j}: \C[\alpha,\gamma]/J_g^- \arrow{r} & \C, &  \qquad \text{ev}^\pm_j(f) = f\left(\pm 4(g-2j),\;0\right).
\end{tikzcd}
\end{equation*}
\vspace{.20cm}

\noindent These are well-defined homomorphisms because $\text{ev}_j^\pm(\zeta_g^-)=\text{ev}^\pm_j(\zeta_{g+1}^-)=\text{ev}_j^\pm(\zeta_{g+2}^-)=0$ for $j$ as above, as is easily verified through the recursion relations for $\zeta_g^-$. Our invertibility test is:

\vspace{.55cm}

\begin{lemma}\label{lemma:invertible}
    If $g$ is odd, $u\in \C[\alpha,\gamma]$ is a unit mod $J_g^-$ if and only if $\text{{\emph{ev}}}_j^\pm(u)\neq 0$ for $j=1,\ldots, (g-1)/2$.
\end{lemma}
\vspace{.40cm}

\begin{proof}
Write $\alpha$ and $\gamma$ for the linear endomorphisms of $\C[\alpha,\gamma]/J_g^-$ defined via multiplication by $\alpha$ and $\gamma$, respectively. These commute and have simultaneous eigenvalues $(\pm 4(g-2j),0)$ where $j=1,\ldots, (g-1)/2$ by Lemma \ref{lemma:munoz}. The element $u$ may be viewed as a linear endomorphism which is a polynomial in the endomorphisms $\alpha$ and $\gamma$. Basic linear algebra says that the eigenvalues of $u$ are given by $u$ evaluated at these pairs of eigenvalues. The condition that $\text{ev}_j^\pm(u)\neq 0$ for the above values of $j$ can now be understood as the non-vanishing of all eigenvalues of $u$.
\end{proof}

\vspace{.45cm}

\noindent A similar test works for $g$ even. With these preliminary remarks out of the way, we now prove a lemma for when $g$ is odd which will easily imply the above claim that $J_g^-=J_{g-1}^-$. We remind the reader that we use the notation $a \equiv_g b$ to mean that $a$ and $b$ are equivalent modulo $J_g^-$.\\

\vspace{.35cm}

\begin{lemma}\label{lemma:minusodd}
For $g\geqslant 1$ odd and $i\geqslant 0$ we have the following:
\begin{itemize}
\item[(i)] Modulo $J_g^-$ there exists a unit $u_{g,i}$ such that $u_{g,i} \gamma^{i}\zeta^-_{g-2i-1} \equiv_g \gamma^{i+1}\zeta^-_{g-2i-2}$.
\item[(ii)] Modulo $J_g^-$ there exists a unit $v_{g,i}$ such that $v_{g,i}\gamma^i \zeta^-_{g-2i-1} \equiv_g  \gamma^{i+1}\zeta^-_{g-2i-3}$.
\end{itemize}
\end{lemma}
\vspace{.30cm}

\begin{proof} We prove (i) and (ii) simultaneously by induction on $i$, fixing $g$. For this reason we simply write $u_i$ and $v_i$ for $u_{g,i}$ and $v_{g,i}$, respectively. We assume the statements hold at $i$ and will prove them at $i+1$. First, substitute $k=g-2i-2$ in (\ref{eq:mo}) and multiply by $\gamma^{i+1}$ to obtain

\vspace{.18cm}
\begin{equation*}
    \alpha\gamma^{i+1}\zeta^-_{g-2i-2}\;\;-\;\;16(g-2i-2)^2\gamma^{i+1}\zeta^-_{g-2i-3}\;\; + \;\; c_1\gamma^{i+2}\zeta^-_{g-2i-4} \;\;=\;\; \gamma^{i+1}\zeta^-_{g-2i-1}
\end{equation*}
\vspace{.02cm}

\noindent in which $c_1 = 2(g-2i-2)(g-2i-3)$. The term on the right side is in $\gamma J_{g-1}^-$ by Lemma \ref{lemma:minuseven} since $g-1$ is even, and $\gamma J_{g-1}^- \subset J_g^-$ implies it is a member of $J_g^-$. Thus 

\vspace{.18cm}
\begin{equation*}
    16(g-2i-2)^2\gamma^{i+1}\zeta^-_{g-2i-3}\;\; \equiv_g \;\; c_1\gamma^{i+2}\zeta^-_{g-2i-4} \;\; + \;\; \alpha\gamma^{i+1}\zeta^-_{g-2i-2}.
\end{equation*}
\vspace{.02cm}

\noindent By the inductive assumption of (i) at $i$, the final term is equivalent to $\alpha u_i \gamma^{i}\zeta^-_{g-2i-1}$. This is in turn equivalent to $\alpha u_i v_i^{-1} \gamma^{i+1}\zeta^-_{g-2i-3}$ by (ii) at $i$. With some minor rearranging this establishes (i) at $i+1$ if we can show that the following element is invertible modulo $J_g^-$:

\vspace{.22cm}
\begin{equation}
    u_{i+1} \; = \; \frac{16(g-2i-2)^2 \; - \; \alpha u_i v_i^{-1} }{2(g-2i-2)(g-2i-3)}, \qquad \quad u_0 \; = \; \frac{8g}{g-1}. \label{eq:uinverse}
\end{equation}
\vspace{.15cm}

\noindent We have included $u_0$ above, which proves the base case of (i) at $i=0$. The lemma holds for $g=1$ by direct inspection, so we may assume $g\geqslant 3$ to justify the denominator in $u_0$. Also, we use the above expression for $u_{i+1}$ only when $g-2i-3>0$, in order to justify the denominator in (\ref{eq:uinverse}). When $g-2i-3=0$, we omit the term $(g-2i-3)$ from the denominator of $u_{i+1}$. This does not affect the form of (i), because for this index the right side of (i) is zero, since $\zeta_{-1}^-=0$. Further, when $g-2i-3<0$, statement (i) is vacuously true since both sides are zero. Setting the invertibility of (\ref{eq:uinverse}) aside for a moment, we consider proving (ii) at $i+1$. Set $k=g-2i-3$ in (\ref{eq:me}) and multiply by $\gamma^{i+1}$ to obtain

\vspace{.18cm}
\begin{equation*}
    \gamma^{i+1}\zeta^-_{g-2i-2}\;\; = \;\;\alpha\gamma^{i+1}\zeta^-_{g-2i-3} \;\;+\;\; c_2\gamma^{i+2}\zeta^-_{g-2i-5}
\end{equation*}
\vspace{.02cm}

\noindent in which $c_2 = 2(g-2i-3)(g-2i-4)$. Again, we use the inductive assumption for (i) and (ii) at $i$ to replace the left hand term with $u_i v_i^{-1} \gamma^{i+1} \zeta^-_{g-2i-3}$. So after rearranging, (ii) at step $i+1$ is done if we can show that the following is invertible modulo $J_g^-$:

\vspace{.22cm}
\begin{equation}
    v_{i+1} \; = \; \frac{u_i v_i^{-1} \; - \; \alpha }{2(g-2i-3)(g-2i-4)}, \qquad \quad v_0 \; = \; \frac{-\alpha}{2(g-1)(g-2)}. \label{eq:vinverse}
\end{equation}
\vspace{.15cm}

\noindent Again, we have included $v_0$ to prove the base case of (ii) at $i=0$. We use the above expression for $v_{i+1}$ when $g-2i-5\geqslant 0$. Note that when $g-2i-5=0$, the right side of (ii) is zero, and when $g-2i-5<0$, (ii) is vacuously true. We are now left with showing that $u_{i+1}$ and $v_{i+1}$ are invertible. For this we strengthen the inductive hypothesis. We begin by making the following observation:

\vspace{.22cm}
\begin{equation*}
    \pm \text{ev}_j^\pm\left(u_0v_0^{-1}\right) \; < \; 0, \qquad \;\; j=1,\ldots, (g-1)/2.
\end{equation*}
\vspace{.12cm}

\noindent We add on to our inductive hypothesis the assumption that $\pm \text{ev}_j^\pm\left(u_kv_k^{-1}\right) < 0$ for $0\leqslant k \leqslant i$ and the above values of $j$. We remark that these expressions evaluated at rational numbers clearly have rational values. With this added hypothesis at $k=i$ it is easy to see from the expressions (\ref{eq:uinverse}) and (\ref{eq:vinverse}) that $\text{ev}_j^\pm\left(u_{i+1}\right) \neq  0$ and $\text{ev}_j^\pm\left(v_{i+1}\right) \neq  0$ for the above $j$. More specifically, we have:

\vspace{.22cm}
\begin{equation}
    \text{ev}_j^\pm\left(u_{i+1}\right) \; > \; 0, \qquad \pm\text{ev}_j^\pm\left(v_{i+1}\right) \; < \; 0. \label{eq:signs}
\end{equation}
\vspace{.02cm}

\noindent For example, in the expression for $u_{i+1}$ evaluated at $(\pm 4 (g-2j),0)$, the numerator is a sum of two positive rational numbers, and the denominator is a positive integer. The inequality for $v_{i+1}$ is deduced similarly. By Lemma \ref{lemma:invertible}, these non-vanishing values imply that $u_{i+1}$ and $v_{i+1}$ are invertible mod $J_g^-$. Further, (\ref{eq:signs}) implies $\pm \text{ev}_j^\pm\left(u_{i+1}v_{i+1}^{-1}\right) < 0$ for the requisite values of $j$, which carries through our added hypothesis to the next step at $k=i+1$, completing the proof. 
\end{proof}

\vspace{.75cm}

\noindent When $i \gg 0$ in the above lemma both sides of (i) and (ii) are zero. Inductively, we obtain that all terms on both sides of (i) and (ii) for $i\geqslant 0$ are zero modulo $J_g^-$. In particular, both of (i) and (ii) at $i=0$ yield the equivalence $\zeta_{g-1}^- \equiv_g 0$, so that $\zeta_{g-1}^-\in J_g^-$. We have established:\\

\vspace{.15cm}

\begin{corollary}
For $g\geqslant 1$ odd, $J_g^- = J_{g-1}^-$.
\end{corollary}

\vspace{.55cm}

\noindent With the relation (\ref{eq:minusevendegree}) and $\deg J^-_0 =0$, this completes the computation of $\deg J_g^-$ for all $g\geqslant 0$:

\vspace{.15cm}

\[
    \deg J_g^- \; = \; \deg J_{g+1}^- \; = \; \frac{g(g+2)}{4}, \qquad g\geqslant 0\; \text{    even.}
\]

\vspace{.25cm}

\noindent To compute $P_t(J^-_g)$ we only need to understand how the degree of $J_g^-$, which is the dimension of $\C[\alpha,\gamma]/J_g^-$, is distributed amongst the four homogeneously $\Z/4$-graded summands. This can be understood by computing the initial ideal of $J_g^-$ with respect to some monomial ordering. In fact, it is rather straightforward from our above analysis to write down a Gr\"{o}bner basis for $J_g^-$.\\

\begin{prop}\label{prop:minusgrobner}
Let $g\geqslant 0$ be even. Under the lexicographical monomial ordering with $\alpha >\gamma$, the following set is a Gr\"{o}bner basis for the ideal $J^-_g = J^-_{g+1}$:

\vspace{.20cm}
\begin{equation}
    \left\{\zeta^-_g, \;\; \gamma\zeta^-_{g-2}, \;\; \gamma^2\zeta^-_{g-4}, \;\; \ldots, \;\; \gamma^{g/2-1}\zeta^-_{2}, \;\; \gamma^{g/2}\right\}.\label{eq:minusgrobner}
\end{equation}
\vspace{.10cm}

\noindent Consequently, the initial ideal of $J_g^-$ is generated by the monomials $\gamma^i \alpha^{g-2i}$ with $0 \leqslant i \leqslant g/2$, and thus a vector space basis for $\C[\alpha,\gamma]/J_g^-$ is represented by the following monomials:

\vspace{.40cm}

\begin{center}
\includegraphics{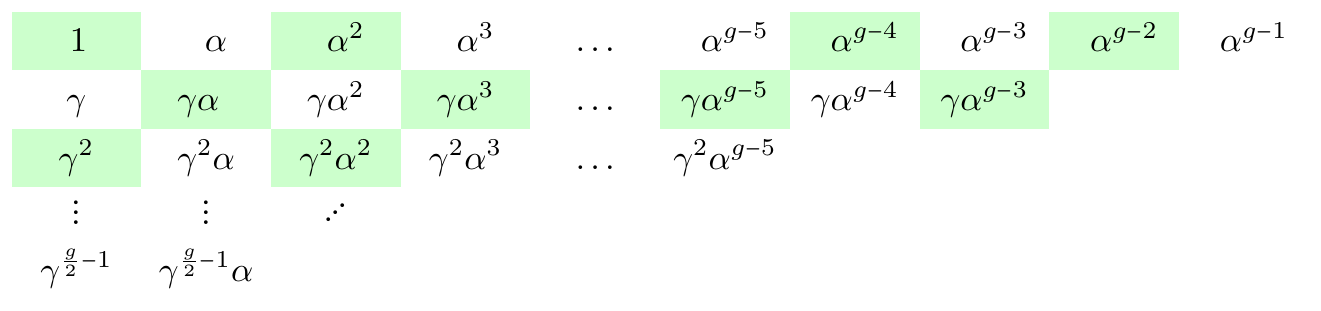}
\end{center}
\end{prop}

\vspace{.05cm}

\begin{proof}
    The elements in (\ref{eq:minusgrobner}) are contained in $J_g^-$ by Lemma \ref{lemma:minuseven}. From (\ref{eq:degalpha}) we find that the initial term of $\gamma^i\zeta_{g-2i}^-$ is the monomial $\gamma^i\alpha^{g-2i}$. It is straightforward to verify that the degree of the initial ideal generated by these monomials agrees with the degree of $J_g^-$ computed above, and thus by a standard result in the theory of Gr\"{o}bner bases, the polynomials $\gamma^i\zeta_{g-2i}^-$ form a Gr\"{o}bner basis, as claimed. It is also a standard result that the monomials not contained in the initial ideal form a vector space basis for the quotient, yielding the final statement. The requisite background for Gr\"{o}bner bases may be found, for example, in \cite[Ch. 2]{monomial}.
\end{proof}

\vspace{.60cm}

\noindent In the above grid of monomials, we have shaded boxes behind monomials with grading $0$ (mod 4). Recalling $\alpha$ and $\gamma$ have grading $2$ (mod 4), all other monomials have grading $2$ (mod 4). This is with the exception of the bottom two monomials: if $g/2$ is odd (resp. even) the monomial $\gamma^{g/2-1}$ (resp. $\gamma^{g/2-1}\alpha$) should be included in this shading. It is easily verified that the number of shaded boxes is always equal to the number of unshaded boxes. We conclude:\\

\begin{corollary}\label{cor:jminus}
For $g\geqslant 0$, the $\Z/4$-graded Poincar\'{e} polynomial for the ideal $J_g^-$ is given by

\vspace{.25cm}
\[  
    P_t(J_g^-)\; =\;  \textstyle{\frac{1}{2}}\lfloor \textstyle{\frac{1}{2}}g \rfloor\left(\lfloor \textstyle{\frac{1}{2}}g \rfloor + 1\right)\cdot \left(1+ t^2\right). 
\]
\end{corollary}

\vspace{.25cm}

\vspace{.25cm}

\subsection{The kernel of $\beta-8$.}
%\noindent\textbf{Computing the kernel of $\beta-8$.}\\
\vspace{.25cm}

\noindent We now turn to the computation of $P_t(J_g^+)$. Here we use the recursive relations (\ref{eq:pe}) and (\ref{eq:me}) instead of (\ref{eq:mo}) and (\ref{eq:po}). The first thing we notice is that now the situation is simple for {\emph{odd}} indices:\\

\vspace{.10cm}

\begin{lemma}\label{lemma:plusodd}
    For $g\geqslant 3$ odd, $J_g^+$ is generated by $\zeta_{g}^+$ and $\gamma J^+_{g-2}$. Further, for $i \geqslant 0$, $\gamma^{i}\zeta_{g-2i}^+\in J_g^+$.
\end{lemma}

\vspace{.50cm}

\noindent The proof is nearly identical to that of Lemma \ref{lemma:minuseven}. The discussion following Lemma \ref{lemma:minuseven} carries over as well. In particular, the recursive relation (\ref{eq:minusevendegree}) holds with $J_g^+$ in place of $J_g^-$, and for $g\geqslant 3$ odd. By direct inspection, $\deg J_1^+=1$, so the recursion relation in this case implies that $\deg J_g^+ = (g+1)^2/4$ for $g\geqslant 1$ odd. We can proceed to write down a Gr\"{o}bner basis for $J_g^+$ when $g$ is odd just as was done for the ideals $J_g^-$, which we will do shortly. 

To handle the case in which $g$ is even, we attempt to mimic the proof of Lemma \ref{lemma:minusodd}. There is an important difference in this situation: multiplication by $\alpha$ on $\C[\alpha,\gamma]/J_g^+$ is {\emph{not}} invertible. This follows from Lemma \ref{lemma:munoz}, which says that $\alpha$ has eigenvalue $0$ on the $+8$ eigenspace of $\beta$ inside $\C[\alpha,\beta,\gamma]/J_g$. Thus when going through the argument of Lemma \ref{lemma:minusodd} we must keep track of $\alpha$ terms more carefully. Proceeding in this fashion yields:\\

\vspace{.25cm}

\begin{lemma}\label{lemma:pluseven}
For $g\geqslant 0$ even and $i\geqslant 0$ we have the following:
\begin{itemize}
\item[(i)] Modulo $J_g^+$ there exists a unit $w_{g,i}$ such that $w_{g,i} \gamma^{i}\zeta^+_{g-2i-1} \equiv_g \gamma^{i+1}\zeta^+_{g-2i-2}$.
\item[(ii)] Modulo $J_g^+$ there exists a unit $x_{g,i}$ such that the following hold:
\[
    \alpha\cdot x_{g,i}\gamma^{i} \zeta^+_{g-2i-1} \equiv_g \gamma^{i+1}\zeta^+_{g-2i-3} \quad \left(i \text{{\emph{ even}}}\right), \qquad x_{g,i}\gamma^{i} \zeta^+_{g-2i-1} \equiv_g \alpha\cdot\gamma^{i+1}\zeta^+_{g-2i-3} \quad \left(i \text{{\emph{ odd}}}\right).
\]
\end{itemize}
\end{lemma}
\vspace{.25cm}

\begin{proof}

As just indicated, the proof is very similar to that of Lemma \ref{lemma:minusodd}. In fact, the manipulations of the relations is exactly the same after replacing ``$+$'' superscripts with ``$-$'' superscripts, changing the sign in front of the $16g^2$ terms, and taking $g$ to be even instead of odd. The resulting recursive formulae for the units differ from the above case only by certain appearances of $\alpha$, which depend on the parity of $i$. Let $\varepsilon(i) = 0$ if $i$ is even and $\varepsilon(i)=1$ if $i$ is odd. Then, writing $w_i$ and $x_i$ in place of $w_{g,i}$ and $x_{g,i}$, the recursion relations we derive are as follows:

\vspace{.22cm}
\begin{equation}
    w_{i+1} \; = \; \frac{-16(g-2i-2)^2 \; - \; \alpha^{2\varepsilon(i)}w_i x_i^{-1} }{2(g-2i-2)(g-2i-3)}, \qquad  x_{i+1} \; = \; \frac{w_i x_i^{-1} \; - \; \alpha^{2(1-\varepsilon(i))}}{2(g-2i-3)(g-2i-4)}. \label{eq:plusinverses}
\end{equation}
\vspace{.15cm}

\noindent We have the initial terms $w_0 = -8g/(g-1)$ and $x_0=-1/(g-1)(g-2)$. The values of $i$ for which $g-2i-3\leqslant 0$ are dealt with just as was done in Lemma \ref{lemma:minusodd}, and we can also assume $g\geqslant 3$, the lower cases holding by direct inspection. All that remains is some way of showing that these equations inductively define invertible elements mod $J_g^+$. For this we use an analogue of Lemma \ref{lemma:invertible}, the invertibility test. For $g$ even, define evaluation maps as follows, for $j=1,\ldots,g/2$:

\vspace{.25cm}
\begin{equation*}
\begin{tikzcd}[column sep=large]
    \text{ev}^\pm_{j}: \C[\alpha,\gamma]/J_g^+ \arrow{r} & \C, &  \qquad \text{ev}^\pm_j(f) = f\left(\pm 4(g-2j)\sqrt{-1},\;0\right).
\end{tikzcd}
\end{equation*}
\vspace{.20cm}

\noindent Then, in the same way we proved Lemma \ref{lemma:invertible}, we see that an element $u\in \C[\alpha,\gamma]/J^+_g$ is invertible if and only if $\text{ev}^\pm_j(u)\neq 0$ for the above values of $j$. We note that $\text{ev}^\pm_j(w_0x_0^{-1}) > 0$. By induction, $(-1)^i\text{ev}^\pm_j(w_ix_i^{-1}) > 0$. More specifically, we find from the recursion formulae (\ref{eq:plusinverses}) that

\vspace{.22cm}
\begin{equation*}
    \text{ev}_j^\pm\left(w_{i}\right) \; < \; 0, \qquad (-1)^{i+1}\text{ev}_j^\pm\left(x_{i}\right) \; > \; 0 \label{eq:signs2}
\end{equation*}
\vspace{.02cm}

\noindent for $j=1,\ldots,g/2$. These non-vanishing evaluations exhibit the invertibility of $w_i$ and $v_i$ at each induction step, and thus complete the proof.
\end{proof}
\vspace{.65cm}

\noindent Before completely describing $J_g^+$ in the style of Proposition \ref{prop:minusgrobner}, we need one more lemma.\\

\vspace{.10cm}

\begin{lemma}\label{lemma:plusjump}
    If $g\geqslant 2$ is even and $u\in \C[\alpha,\gamma]$, then $\gamma^2u\in J_g^+$ implies $u\in J_{g-4}^+$.
\end{lemma}
\vspace{.35cm}

\begin{proof}
    We first use the recursion relations to rewrite a set of generators for $J_g^+$. Define:
    
    \vspace{-.05cm}
\begin{align*}
    \xi_1 \; & := \; 8g \zeta_{g-1}^+ + (g-1)\gamma\zeta_{g-2}^+ \\
    \xi_2 \; & := \; \alpha\gamma\zeta^+_{g-2} - 16g(g-2)\gamma\zeta^+_{g-3}\\
    \xi_3 \; & := \; \gamma\zeta^+_{g-1} 
\end{align*}
\vspace{-.05cm}

\noindent Then $J_g^+ = (\xi_1,\xi_2,\xi_3)$. Indeed, $\xi_1$ is a rational multiple of $\zeta^+_{g+1} -\alpha\zeta_{g}^+$, while $\xi_2$ is a rational multiple of $\alpha\xi_1-8g\zeta^+_g$ and $\xi_3$ is a rational multiple of $\zeta^+_{g+2}-\alpha\zeta^+_{g+1}$. Thus if $\gamma^2 u \in J_g^+$ then $\gamma^2u = f_1\xi_1 + f_2\xi_2 + f_3\xi_3$ where each $f_i\in \C[\alpha,\gamma]$. Since $\xi_2$ and $\xi_3$ are multiples of $\gamma$ and $\xi_1$ is not, we must have that $f_1$ is a multiple of $\gamma$. So we may write $\gamma u  =  f_1' \xi_1 + f_2\xi'_2 + f_3\xi'_3$ in which $f_1'\gamma = f_1$, $\xi'_2\gamma = \xi_2$ and $\xi_3'\gamma =\xi_3$. To simplify things, we may set $\xi'_1 := \gamma\zeta^+_{g-2}$, which is a scalar combination of $\xi_1$ and $\xi_3'$, and we may then write $\gamma u  =  f_1' \xi'_1 + f'_2\xi'_2 + f'_3\xi'_3$ where now we have

\vspace{-.05cm}
\begin{align*}
    \xi'_1 \; & = \; \gamma\zeta^+_{g-2} \\
    \xi'_2 \; & = \; \alpha\zeta^+_{g-2} - 16g(g-2)\zeta^+_{g-3}\\
    \xi'_3 \; & = \; \alpha\zeta^+_{g-2} + 16(g-2)^2\zeta^+_{g-3} + 2(g-2)(g-3)\gamma\zeta_{g-4}^+
\end{align*}
\vspace{-.05cm}

\noindent The expression for $\xi_3'$ is just the recursion expansion for $\zeta_{g-1}^+$. Write $f_i'=\gamma q_i + r_i$ for $i=2,3$ where $r_i$ is a polynomial in $\alpha$. Since the leading term of both $\xi_2'$ and $\xi_3'$ is $\alpha^{g-1}$ and $\xi_1'$ is a multiple of $\gamma$, we see that in order for $\gamma u  =  f_1' \xi'_1 + f'_2\xi'_2 + f'_3\xi'_3$ to hold we must have $r_2=-r_3$.  Now $r_2=-r_3$ implies that the $\alpha\zeta_{g-2}^+$ terms in $\xi_2'$ and $\xi_3'$ cancel, and the remaining term in $\gamma u$ which is not a multiple of $\gamma$ is simply $-16 g(g-2)r_2\zeta_{g-3}^+$ plus $16(g-2)^2r_3\zeta_{g-3}^+$, which is $32(g-1)(g-2)r_3\zeta_{g-3}^+$. Thus:

\vspace{-.05cm}
\begin{align*}
   \gamma u  \; & \phantom{:}= \; p_1 \zeta_{g-3}^+\; + \; \gamma p_2 \\
    p_1 \; & := \; 32(g-1)(g-2)r_3\\
    p_2 \; & := \; f_1'\zeta_{g-2}^+  + q_2\xi_2' + q_3\xi'_3
\end{align*}
\vspace{-.05cm}

\noindent Now since $\zeta_{g-3}^+$ is not a multiple of $\gamma$, in fact $p_1=\gamma p_1'$ for some $p_1'$. We then have $u =  p_1' \zeta_{g-3}^+ + p_2$. Since $\zeta_{g-3}^+$ and $p_2$ are members of $J_{g-4}^+$, we are done.
\end{proof}

\vspace{.60cm}

\noindent With these lemmas we can now prove an analogue of Proposition \ref{prop:minusgrobner}.\\

\vspace{.45cm}

\begin{prop}\label{prop:plusgrobner}
Let $g\geqslant 1$ be odd. Under the lexicographical monomial ordering with $\alpha >\gamma$, the following set is a Gr\"{o}bner basis for the ideal $J^+_g$:

\vspace{.20cm}
\begin{equation*}
    \left\{\zeta^+_g, \;\; \gamma\zeta^+_{g-2}, \;\; \gamma^2\zeta^+_{g-4}, \;\; \ldots, \;\; \gamma^{(g-1)/2}\zeta^+_{1}, \;\; \gamma^{(g+1)/2}\right\}.
\end{equation*}
\vspace{.10cm}

\noindent Consequently, the initial ideal of $J_g^+$ is generated by the monomials $\gamma^i \alpha^{g-2i}$ with $0 \leqslant i \leqslant (g+1)/2$ and thus a basis for $\C[\alpha,\gamma]/J_g^+$ is represented by the following monomials:

\vspace{.40cm}

\begin{center}
\includegraphics{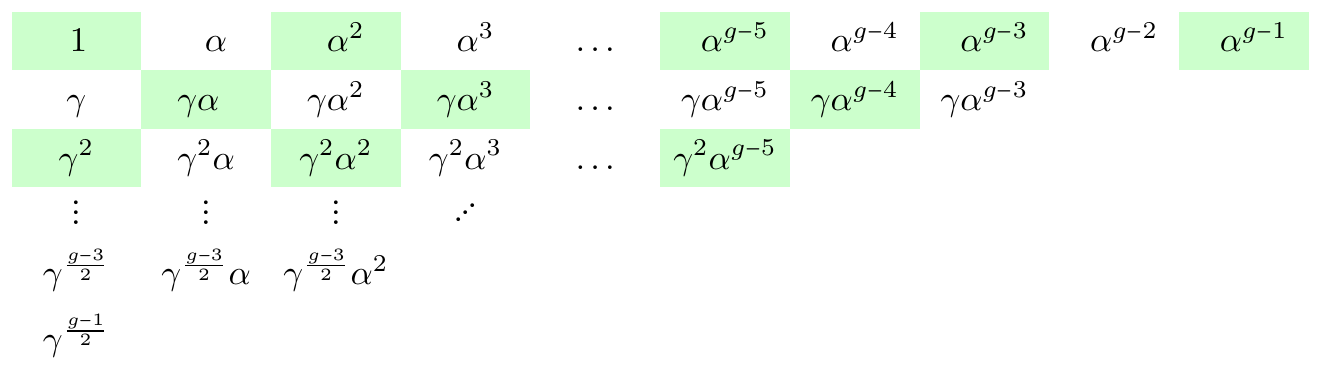}
\end{center}

\vspace{.20cm}

\noindent The even index cases are as follows. If $g +1 \equiv 0 \text{ \emph{(mod 4)}}$ then $J_{g+1}^+ = J_{g}^+$. Otherwise, we have:

\vspace{.20cm}
\[
\begin{array}{lll}
    g + 1 \equiv 2 \text{ \emph{(mod 4)}}: \quad & \deg J_{g+1}^+ - \deg J_{g}^+ = 1, & \quad J_{g}^+ \; = \; \left( J_{g+1}^+, \; \gamma^{(g+1)/2}\right).
\end{array}
\]
\vspace{.10cm}

\noindent Thus a vector space basis for $\C[\alpha,\gamma]/J_{g+1}^+$ when $g + 1 \equiv 2$ {\emph{(mod 4)}} is given by the above grid of monomials with the addition of $\gamma^{(g+1)/2}$.

\end{prop}
\vspace{.65cm}

\noindent As before, we have shaded boxes behind the terms of degree $0$ (mod 4), and the bottom four terms should be shaded in this fashion according to the parity of $(g-1)/2$.\\

\begin{proof}
    The statement regarding $J_g^+$ for $g$ odd follows from Lemma \ref{lemma:plusodd} and the argument in Proposition \ref{prop:minusgrobner}. Note in particular that the inclusion $\gamma^{(g+1)/2}\in J_g^+$ follows inductively from the base case $\gamma\in J_1^+$ and Lemma \ref{lemma:plusodd}. It remains to prove the final statements regarding $J^+_{g+1}$ for $g+1$ even. Henceforth we set $k:=g+1$, so that $k$ is even. First, consider the exact sequence
    
    \vspace{.10cm}

\begin{equation*}
\begin{tikzcd}[column sep=large]
   0 \arrow{r} & \text{ker}(\alpha) \arrow{r} & \C[\alpha,\gamma]/J_{k}^+ \arrow{r}{\alpha} &  \C[\alpha,\gamma]/J_{k}^+ \arrow{r} & \text{coker}(\alpha) \arrow{r} & 0 
\end{tikzcd}
\end{equation*}
\vspace{.10cm}

\noindent in which the map $\alpha$ is multiplication by $\alpha$. The cokernel may be identified with $\C[\gamma]/I_k^+$ in which $I_k^+$ is defined as was $J_k^+$ by setting $\alpha=0$. It is easily verified that $I_k^+$ is the principal ideal generated by $\gamma$, and thus $\deg I_k^+ =1$. On the other hand, taking $i \gg 0$ and recalling that by definition $\zeta_{r}^+$ vanishes for $r <0$, we obtain the following inclusions from Lemma \ref{lemma:pluseven} (ii):
    
    \vspace{.15cm}
    \begin{equation}
        \alpha\gamma^{i}\zeta^+_{k-2i-1} \in J_k^+ \quad \text{ ($i$ even)}, \qquad \quad \gamma^{i}\zeta^+_{k-2i-1} \in J_k^+ \quad \text{ ($i$ odd)}.\label{eq:relnsplus}
    \end{equation}
    \vspace{.08cm}
    
    \noindent In particular, setting $i=0$ yields $\alpha \zeta_{k-1}^+ \in J_k^+$. This implies that $\alpha J_{k-1}^+ \subset J_k^+$, which in turn implies that $J_{k-1}^+/J_k^+ \subset \text{ker}(\alpha)$. Thus from the above exact sequence we obtain:\\

\begin{equation}
   0\;  \leqslant \; \deg J_{k}^+  \; - \; \deg J_{k-1}^+  \; \leqslant \; 1, \qquad k\geqslant 2 \text{ and even}.\label{eq:plusevendegree}
\end{equation}
\vspace{.10cm}

\noindent Now suppose $k\equiv 0$ (mod 4). Then (\ref{eq:relnsplus}) with $i=k/2-1$ yields $\gamma^{k/2-1} \alpha \in J_k^+$, in which we have identified $\zeta_1^+=\alpha$. In other words, $\gamma^{k/2-1}$ represents an element in $\text{ker}(\alpha)$. On the other hand, $\gamma^{k/2-1}\notin J_k^+$ follows from $\gamma^{k/2-1}\notin J^+_{k-1}$, which we know since $k-1$ is odd. This gives a generator for the 1-dimensional space $\text{ker}(\alpha)$ which is not contained in the subspace $J_{k-1}^+/J^+_k$. We conclude that $J_{k}^+ = J^+_{k-1}$. In other words, the difference of degrees in (\ref{eq:plusevendegree}) is equal to zero. This completes the claim in the proposition regarding the case $k = g+1 \equiv 0$ (mod 4).

For the case in which $k = g+1 \equiv 2$ (mod 4) we claim that the difference of degrees in (\ref{eq:plusevendegree}) is instead equal to $1$. To achieve this it suffices to show that $\gamma^{k/2}$, which is a member of $J_{k-1}^+$, is not a member of $J_k^+$. We use induction on $k$ over $k\geqslant 2$ with $k\equiv 2$ (mod 4). By direct inspection, $\gamma \notin J_2^+$. Suppose the result holds up to and including the index $k-4$, and suppose for contradiction that $\gamma^{k/2}\in J_k^+$. Then by Lemma \ref{lemma:plusjump}, $\gamma^{k/2} = \gamma^2\gamma^{(k-4)/2} \in J_k^+$ implies $\gamma^{(k-4)/2}\in J_{k-4}^+$, contradicting the inductive hypothesis, and completing the proof.
\end{proof}

\vspace{.65cm}

\begin{corollary}\label{cor:jplus}
For $g\geqslant 0$, the $\Z/4$-graded Poincar\'{e} polynomial for the ideal $J_g^+$ is given by

\vspace{.25cm}
\[  
    P_t(J_g^+)\;\;=\;\; \begin{cases}\quad\left\lceil \frac{1}{8}(g+1)^2 \right\rceil\; \; + \;\;\left\lfloor \frac{1}{8}(g+1)^2 \right\rfloor \cdot t^2 & \quad g \text{ odd}\\
     & \\
    \quad \;\;\;\;\;\;\;\; \left\lceil \frac{1}{8}g^2 \right\rceil \;\;\;\, + \;\;\;\left\lceil \frac{1}{8}g^2 \right\rceil \cdot t^2 & \quad g \text{ even}\end{cases}
\]
\end{corollary}

\vspace{.65cm}

\noindent This follows from the proposition by counting the number of shaded boxes for the constant term, and counting the rest of the monomials for the coefficient in front of $t^2$.

\newpage

\subsection{The betti numbers}

\vspace{.20cm}

We now have all the information we need to compute the $\Z/4$-graded Poincar\'{e} polynomial for the framed instanton homology via (\ref{eq:poincarepolynomial}), which leads to Theorem \ref{thm:main}. First, we write down the polynomial for the invariant part, for the plus and minus parts, which have essentially been computed:

\vspace{.35cm}

\begin{equation}
    P_t\left( I^\#_\text{inv}(\Sigma\times S^1)^\pm_w \right)    \; = \;  (1+t^3)P_t\left(\text{ker}\left(\beta \pm 8 \; \rotatebox[origin=c]{-90}{$\circlearrowright$} \; I_\text{inv}(\Sigma\times S^1)'_w \right)\right)  \;  = \; (1+t^3)P_t(J^\mp_g).\label{eq:poincarerel}
\end{equation}

\vspace{.45cm}

\noindent Here we use the notation $f\,\rotatebox[origin=c]{-90}{$\circlearrowright$}\, V$ to mean that a map $f$ is acting as an endomorphism on the vector space $V$. We again remark that the factor $(1+t^3)$ is present because the framed group on the left is the mapping cone of $\beta \pm 8$, so is isomorphic to the kernel plus cokernel of $\beta \pm 8$ acting on the invariant part of Mu\~{n}oz's ring; the cokernel has the same dimension as the kernel, but has an overall grading shift of $-1 \equiv 3 $ (mod 4). From Corollaries \ref{cor:jminus} and \ref{cor:jplus} we deduce the following:
\vspace{.25cm}

\begin{equation}
    P_t\left( I^\#_\text{inv}(\Sigma\times S^1)^\pm_w \right) \; = \; 
    \begin{cases}
        \;\;\;\frac{1}{2}\lfloor \frac{1}{2}g \rfloor\left(\lfloor \frac{1}{2}g \rfloor + 1\right)\cdot(1+t+t^2+t^3) & \pm = +, \quad \text{all } g\\
        &\\
        \;\;\;\left\lceil \frac{1}{8}g^2 \right\rceil\cdot(1+t+t^2+t^3) & \pm = -, \quad g \text{ even}\\
        &\\
        \;\;\;\left\lceil \frac{1}{8}(g+1)^2 \right\rceil\cdot (1+t^3) \; + \; \left\lfloor \frac{1}{8}(g+1)^2 \right\rfloor\cdot (t+t^2) \;\;\; & \pm = -, \quad g \text{ odd}
    \end{cases}\label{eq:poincareinv}
\end{equation}

\vspace{.60cm}

\noindent To move beyond the invariant part, we sum over the tensor powers of primitive components as in (\ref{eq:poincarepolynomial}). We can compute the polynomials in this way first for the kernel of $\beta \pm 8$ on Mu\~{n}oz's ring, and later multiply by $(1+t^3)$ for the framed group. Recalling from the introduction the definition

\vspace{.15cm}

\[
    s_i(g) \; := \; \sum_{\substack{ 0 \leqslant k < g  \\ k \equiv i (\text{mod }4)}} {2g \choose k}
\]

\vspace{.30cm}

\noindent and setting $\varepsilon = 0$ if $g$ is even, and $\varepsilon = 1$ if $g$ is odd, the $\Z/4$-graded Poincar\'{e} polynomials of the kernels of $\beta \pm 8$ on the totality of Mu\~{n}oz's ring are given by the following:\\

\vspace{.60cm}

\begin{prop}\label{prop:poincareker}
    With $P_t\left(\text{{\emph{ker}}}\left(\beta \pm 8  \; \rotatebox[origin=c]{-90}{$\circlearrowright$} \; I(\Sigma\times S^1)'_w \right)\right) \; = \; i_0^\pm + i^\pm_1\cdot t + i^\pm_2 \cdot t^2 + i^\pm_3\cdot t^3$ and $j\in \{0,1\}$:
    
    \vspace{.20cm}
    \[
    \begin{array}{rlrl}
         i_{0+\varepsilon}^+ \,= \, i_{2+\varepsilon}^+ &  \hspace{-.2cm}=\, \displaystyle \frac{g^2-g}{4(2g-1)}{2g\choose g} &  \quad i_{0+\varepsilon}^-  \,=\,  i_{2+\varepsilon}^- &  \hspace{-.2cm}=\,  \displaystyle \frac{g^2+3g-2}{4(2g-1)}{2g \choose g}-2^{g-2}\left(1+2^{g-1}\right) \\
      &\phantom{r}&&\\
      &\phantom{r}&&\\
      i_{1+\varepsilon}^+ \,= \, i_{3+\varepsilon}^+ & \hspace{-.2cm}=\,  \displaystyle \frac{g^2}{4(2g-1)}{2g \choose g} -  2^{2g-3} & i_{\varepsilon+2j-1}^-  &  \hspace{-.2cm}= \, \displaystyle \frac{g^2}{4(2g-1)}{2g\choose g} + (-1)^{j} \left(s_{1-\varepsilon}(g)-2^{2g-3}\right)
    \end{array}
    \]

\end{prop}

\vspace{.35cm}

\begin{proof}
	The computations are routine manipulations of series involving binomial coefficients. We only compute the particular case of $i^-_0$ when $g\equiv 1$ (mod 4) for illustration, and leave the rest to the interested reader. The number $i_0^-$ is equal to the constant coefficient in the polynomial

\vspace{.25cm}
\begin{equation*}
    \sum_{k=0}^g \left[{2g \choose k} - {2g \choose k-2} \right]\cdot t^{3k}P_t\left(J^+_{g-k}\right).
\end{equation*}
\vspace{.18cm}

\noindent Recall that we view our polynomials as elements in $\Z[t]/(t^4-1)$. Thus ``constant coefficient'' is synonomous with the sum of the coefficients appearing in front of the monomials $t^{4\ell}$ for $\ell\geqslant 0$. From Corollary \ref{cor:jplus} we identify two kinds of contributions: those involving the constant coefficient of $P_t(J_{g-k}^+)$ multiplied against $t^{3k}$ when $k \equiv 0$ (mod 4), and those involving the $t^2$ coefficient of $P_t(J_{g-k}^+)$ multiplied against $t^{3k}$ when $k\equiv 2$ (mod 4). Thus we may write

\vspace{.25cm}
\begin{equation*}
    i^-_0 \; \; = \;\; \sum_{\substack{ 0 \leqslant k \leqslant g  \\ k \equiv 0 (\text{mod }4)}} \left[{2g \choose k} - {2g \choose k-2} \right]\cdot\left\lceil \textstyle{\frac{1}{8}}(g-k+1)^2 \right\rceil \;\; + \;\;  \sum_{\substack{ 0 \leqslant k \leqslant g  \\ k \equiv 2 (\text{mod }4)}} \left[{2g \choose k} - {2g \choose k-2}\right]\cdot\left\lfloor \textstyle{\frac{1}{8}}(g-k+1)^2 \right\rfloor.
\end{equation*}
\vspace{.18cm}

\noindent With the assumption that $g\equiv 1$ (mod 4), and the mod 4 congruence conditions on $k$ in these sums, we may remove the ceiling and floor functions to obtain the following:

\vspace{.25cm}
\begin{equation*}
    i^-_0 \; \; = \;\; \sum_{\substack{ 0 \leqslant k \leqslant g  \\ k \equiv 0 (\text{mod }2)}} \left[{2g \choose k} - {2g \choose k-2} \right]\cdot  {\textstyle{\frac{1}{8}}}(g-k+1)^2 \;\; + \;\;  \sum_{\substack{ 0 \leqslant k \leqslant g  \\ k \equiv 0 (\text{mod }4)}} \left[{2g \choose k} - {2g \choose k-2}\right]\cdot\textstyle{\frac{1}{2}}.
\end{equation*}
\vspace{.18cm}

\noindent Next, we simplify the first term on right side by collecting terms in front of common binomial coefficients, and use the definition of $s_i(g)$ to rewrite the second term, to obtain the following:

\vspace{.28cm}
\begin{equation*}
    i^-_0 \; \; = \;\; \sum_{\substack{ 0 \leqslant k \leqslant g  \\ k \equiv 0 (\text{mod }2)}}{2g \choose k}\cdot  {\textstyle{\frac{1}{2}}}(g-k) \;\;+\;\; \frac{1}{2}s_0(g) \;\; - \;\; \frac{1}{2}s_2(g).
\end{equation*}
\vspace{.18cm}

\noindent Using ${2g\choose k} = {2g\choose 2g-k}$ we see that $s_0(g) + s_2(g)$ is the sum of binomial coefficients ${2g \choose k}$ with $k\equiv 0$ (mod 4) from $k=0$ to $k=2g$. For $g \geqslant 1$ we have the general formula

\vspace{.30cm}
\begin{equation*}
    \sum_{\substack{ 0 \leqslant k \leqslant 2g  \\ k \equiv 0 (\text{mod }4)}}{2g \choose k} \;\; = \;\; \frac{1}{4}\left( 2^{2g} + (-1)^{\lfloor g/2 \rfloor}\left(1+(-1)^g\right)2^g  \right),
\end{equation*}
\vspace{.18cm}

\noindent obtained, for example, by expanding $\frac{1}{2}(1+\sqrt{-1})^{2g} + \frac{1}{2}(1-\sqrt{-1})^{2g}$ using the binomial theorem on the one hand, and writing it in complex polar coordinates on the other. In our situation this yields the relation $s_0(g) +s_2(g) = 2^{2g-2}$. Thus we now have the following:

\vspace{.50cm}

\begin{equation}
    i^-_0 \; \; = \;\;\frac{g}{2}\sum_{\substack{ 0 \leqslant k \leqslant g  \\ k \equiv 0 (\text{mod }2)}} {2g \choose k} \;\; - \;\; \frac{1}{2}\sum_{\substack{ 0 \leqslant k \leqslant g  \\ k \equiv 0 (\text{mod }2)}} {2g \choose k} k \;\; + \;\; s_0(g) \;\; - \;\; 2^{2g-3}.\label{eq:someeq}
\end{equation}
\vspace{.18cm}

\noindent The first term on the right is easily dealt with: the sum of binomial coefficients behind $g/2$ is exactly half the sum of all binomial coefficients ${2g \choose k}$ with $k$ even from $k=0$ to $k=2g$, which itself is equal to $2^{2g-1}$. Thus the first term is equal to $g 2^{2g-3}$. The second term on the right is computed similarly, upon using the identity ${2g \choose k}k = {2g-1\choose k-1}2g$. We first rewrite it as follows:

\vspace{.25cm}
\begin{equation}
   -\frac{1}{2} \sum_{\substack{ 0 \leqslant k \leqslant g  \\ k \equiv 0 (\text{mod }2)}} {2g \choose k} k \; \; = \;\; -g \sum_{\substack{ 0 \leqslant k \leqslant g  \\ k \equiv 0 (\text{mod }2)}} {2g-1 \choose k-1}  \;\; = \;\; -g\sum_{\substack{ 0 \leqslant k \leqslant g  \\ k \equiv 0 (\text{mod }2)}} {2g-2 \choose k-1} + {2g-2\choose k-2}.\label{eq:term2}
\end{equation}
\vspace{.18cm}

\noindent The final expression behind the factor of $-g$ is the sum of binomial coefficients ${2g-2\choose \ell}$ for $\ell$ from $\ell = 0$ to $\ell = g-2$. With some compensation of the ``middle term,'' this is just half the sum of binomial coefficients ${2g-2\choose \ell}$ for $\ell$ from $\ell = 0$ to $\ell =2g-2$. Precisely, the expression (\ref{eq:term2}) is equal to $-g2^{2g-3}+ \frac{g}{2}{2g-2\choose g-1}$. Plugging the computations of these first two terms into (\ref{eq:someeq}) yields:

\vspace{.25cm}
\begin{equation*}
    i^-_0 \; \; = \;\; g2^{2g-3} \;\;+ \;\; \left(\frac{g}{2}{2g-2\choose g-1} - g 2^{2g-3} \right) \;\; + \;\; s_0(g) \;\; - \;\; 2^{2g-3}.
\end{equation*}
\vspace{.18cm}

\noindent Finally, using the relation ${2g-2\choose g-1} = \frac{g}{2(2g-1)}{2g \choose g}$, from this we obtain the expression stated in the proposition for $i^-_{1+2j-\varepsilon}$ with $j=0$ and $\varepsilon=1$. The computations for the other numbers $i_\ell^\pm$ involve the same types of manipulations, and no more.
\end{proof}
\vspace{.95cm}

\noindent From Proposition \ref{prop:poincareker} and (\ref{eq:poincarerel}) we gather that $b_\ell^\pm = i_\ell^{\pm} + i_{\ell +1}^{\pm}$, and obtain the following:\\

\vspace{.40cm}

\begin{corollary}
	 Let $P_t\left(I^\#(\Sigma\times S^1)^\pm_w \right) \; = \; b_0^\pm + b^\pm_1\cdot t + b^\pm_2 \cdot t^2 + b^\pm_3\cdot t^3$. For the ``$+$'' case we have:

\vspace{.20cm}
\[
b_0^+ \, = \, b^+_1 \, = \, b_2^+ \, = \, b_3^+ \,=\, \frac{g}{4}{2g\choose g} - 2^{2g-3}.
\]
\vspace{.20cm}

\noindent With $\varepsilon =0$ for $g$ even and $\varepsilon =1$ for $g$ odd as above, for the ``$-$'' case we have:
\vspace{.20cm}

\begin{align*}
b_{0+\varepsilon}^- \, = \, b^-_{1+\varepsilon} \, & = \,  \frac{g+2}{4}{2g\choose g} - s_{1-\varepsilon}(g) - 2^{g-2},\\
&\\
	b_{2+\varepsilon}^- \, = \, b^-_{3+\varepsilon} \, & = \, \frac{g+2}{4}{2g\choose g} + s_{1-\varepsilon}(g) - 2^{g-2}\left(1+2^g\right).
\end{align*}

\end{corollary}

\vspace{.55cm}

\noindent Theorem \ref{thm:main} now follows from this corollary by adding together the plus and minus polynomials. We also remark that formula (\ref{eq:invdim}) for the dimension of the invariant part of the framed instanton homology follows in the same way from (\ref{eq:poincareinv}).\\

\newpage

%%%%%%%%%%%
\subsection{Comparison to the cohomology of the critical set}\label{sec:comparison}

\vspace{.30cm}

Having finished the proof of Theorem \ref{thm:main}, we briefly return to the comparison between the framed instanton homology and the singular cohomology of the framed moduli space $N^g_0$ that was described in Section \ref{sec:background}. We can see how the above arguments adapt, and in fact greatly simplify, when working within the ring $H^\ast(N^g)$ in place of $I(\Sigma\times S^1)_w'$. From the Gysin exact sequence (\ref{eq:gysin}), the cohomology $H^\ast(N_0^g)$ is the mapping cone of multiplication by $\beta$ on the ring $H^\ast(N^g)$. Thus we may first consider the kernel of $\beta$. For this we proceed just as we did at the beginning of Section \ref{sec:framed}, and define ideals $J_g'\subset \C[\alpha,\gamma]$ recursively by setting $J_g' = (\zeta_g',\zeta_{g+1}',\zeta_{g+2}')$, where $\zeta_g' = 0$ for $g<0$, $\zeta_0' =1$, and

\vspace{.18cm}
\[
    \zeta_{g+1}' \; = \; \alpha\zeta_g' \; + \; 2\gamma g (g-1) \zeta_{g-2}'
\]
\vspace{.10cm}

\noindent for $g\geqslant 0$. This is the result of plugging $\beta=0$ into the recursively defined presentation for $H^\ast(N^g)$ that was mentioned in Section \ref{sec:ring}. In particular, if we only care about $\Z/4$-gradings, then the sum of two copies of the vector space $\C[\alpha,\gamma]/J_g'$, one of them shifted in grading by $3$ (mod 4), is isomorphic to the invariant part of $H^\ast(N^g_0)$. More generally, the integer graded betti numbers are dictated by the gradings in the Gysin sequence (\ref{eq:gysin}).

The first thing to notice is that the short manipulation of Lemma \ref{lemma:minuseven} adapts to show that for all $g\geqslant 2$ the ideal $J_g'$ is generated by $\zeta_g'$ and $\gamma J_{g-2}'$. Inductively, we obtain that $\gamma^i\zeta'_{g-2i}\in J_g'$ for $i\geqslant 0$, and since $\gamma\in J_1'$ we also inductively have $\gamma^{i} \in J_g'$ for $i=\lceil g/2 \rceil$. It is easily verified that the discussion following Lemma \ref{lemma:minuseven} also adapts, with the kernel of multiplication by $\gamma$ identified as $J_{g-2}'/J_g'$. The following is an analogue of Proposition \ref{prop:minusgrobner}, and the proof uses the same basic theory:\\

\vspace{.25cm}

\begin{prop}
    Under the lexicographical monomial ordering with $\alpha >\gamma$, a Gr\"{o}bner basis for $J'_g$ is:

\vspace{.20cm}
\begin{equation*}
    \left\{\zeta'_g, \;\; \gamma\zeta'_{g-2}, \;\; \gamma^2\zeta'_{g-4}, \;\; \ldots, \;\; \gamma^{\lfloor g/2\rfloor}\zeta'_{g-2\lfloor g/2\rfloor}, \;\; \gamma^{\lceil g/2\rceil}\right\}.
\end{equation*}
\vspace{.10cm}

\noindent Consequently, a vector space basis for $\C[\alpha,\gamma]/J_g'$ is represented by the monomials listed in the grid of Lemma \ref{prop:minusgrobner} for $g$ even, and  the monomials listed in the grid of Lemma \ref{prop:plusgrobner} for $g$ odd.
\end{prop}

\vspace{.60cm}

\noindent We remark that when $g$ is even the last two elements of the Gr\"{o}bner basis above are equal. We emphasize that hardly any work is required in establishing this proposition, and all of the little complications that arose in the proof of Theorem \ref{thm:main} disappear in this setting. Also, Newstead's betti numbers (\ref{eq:newsteadbetti}) can be recovered using this proposition and a formula similar to (\ref{eq:poincarepolynomial}). Note that if $\C[\alpha,\gamma]/J'_g$ is viewed as a $\Z/4$-graded vector space, then in fact we have

\vspace{.35cm}
\begin{equation*}
    P_t\left(J'_g\right) \;\; = \;\;\begin{cases} \;\;\;\;\; P_t\left(J^-_g\right) & \;\;\;\; g\text{ even}, \\ & \\\;\;\;\;\; P_t\left(J^+_g\right) &\;\;\;\; g\text{ odd}.\end{cases}
\end{equation*}
\vspace{.40cm}

\noindent At this level we can see the rank inequalities $n_i \geqslant b_i$ from the introduction as arising from the more basic inequalities between the coefficients of $P_t(J_g^-)$ and $P_t(J_g^+)$ for different parities of $g$: when $g$ is even, the former's coefficients are larger, while the opposite is true for $g$ odd. Further, we can proceed just as in the previous section to compute the mod 4 graded betti numbers of $N_0^g\sqcup N_0^g$. If we let the numbers $i^\pm_\ell$ be as computed in Proposition \ref{prop:poincareker}, and as before let $\varepsilon=0$ if $g$ is even and $\varepsilon = 1$ if $g$ is odd, then we have the following formulae for $n_\ell = \dim H_\ell(N_0^g\sqcup N_0^g)$ with $\ell\in \Z/4$:

\vspace{.10cm}

\begin{align*}
	n_{0+\varepsilon} \, = \, n_{1+\varepsilon} \,  = \, 2\left(i_\varepsilon^+ \,+\, i^-_{1+\varepsilon}\right),\qquad \quad n_{2+\varepsilon} \, = \, n_{3+\varepsilon} \,  = \,  2\left(i_{\varepsilon}^+\, + \, i^-_{3+\varepsilon}\right).
\end{align*}

\vspace{.38cm}

\noindent  These can also be computed directly from Newstead's formula (\ref{eq:newsteadbetti}) for the betti numbers of $H^\ast(N^g_0)$. The numbers on the right hand sides are given by the formulae in Proposition \ref{prop:poincareker}. Also note that from these expressions it is easy to read off that the total rank of $H^\ast(N^g_0)$, which is half the sum of these four numbers, is equal to $g{2g \choose g}$.

\vspace{.60cm}

\appendix
\section{Appendix: the twisted Gysin sequence}

\vspace{.05cm}

In this appendix we sketch the derivation of the twisted Gysin sequence (\ref{eq:twistedgysin}). Although this is essentially part of \cite[Thm 1.3]{scaduto}, we include it here because the proof is very simple assuming Fukaya's connect sum theorem, and also because the formulation is slightly different from that of loc. cit. Suppose that $(Y,w)$ is a non-trivial admissible pair as in the introduction. Recall that $I^\#(Y)_w$ may be defined as the instanton homology group $I(Y\# T^3)_v'$ where $v$ is equal to $w$ over $Y$ and $t=w|_{T^3}$ is non-trivial over the 3-torus. This latter $\Z/4$-graded group is in turn a quotient of Floer's $\Z/8$-graded group $I(Y\# T^3)_v$ by a degree four involution.

The idea of the version of Fukaya's theorem under consideration is as follows. The chain complex $C(Y)_w$ for Floer's $\Z/8$-graded group $I(Y)_w$ is generated by a finite set of isolated irreducible projectively flat connections on a unitary bundle $E$ with $\det E = w$, modulo determinant one gauge transformations. In actuality, this is true after a suitable perturbation, which we fix. In the particular case in which $Y$ is a 3-torus, there is no need to perturb, and there are exactly two generators, say $\rho_1$ and $\rho_2$, as one can verify via the holonomy correspondence.

With a perturbation for $Y$ fixed as above, the space of (perturbed) projectively flat connections on the relevant bundle over $Y\# T^3$ is a disjoint union of copies of $SO(3)$, one for each pair $(\rho, \rho_i)$ where $i\in \{1,2\}$ and $\rho$ is a generator for $C(Y)_w$. The parameter space $SO(3)$ may be thought of as the different ways of gluing $\rho$ and $\rho_i$ together.

Fukaya's Bott-Morse type spectral sequence says that there is a spectral sequence whose starting page is the homology of this space which converges to $I(Y\# T^3)_v$. The group $C(T^3)_t$ is of rank two, generated by $\rho_1$ and $\rho_2$, the gradings of which differ by four. We see that the $E^1$-page is

\vspace{.15cm}
\[
E^1\; = \; \Big(C(Y)_w \otimes C(T^3)_t\Big)\otimes H_\ast(SO(3)).
\]
\vspace{.04cm}

\noindent Here it is important to mention that our coefficients are the complex numbers, so that $H_\ast(SO(3))$ is a rank two vector space. Fukaya identifies the higher differentials in terms of the $u$-maps. The result is that the differential on the $E^1$-page computing the $E^\infty$-page is given by

\vspace{.15cm}
\[
    \partial \otimes 1\otimes \varepsilon \; + \;  \left(u\otimes 1 \, + \, 1\otimes \tau\right)\otimes \mu
\]
\vspace{.04cm}

\noindent in which $\partial$ is the differential on Floer's complex $C(Y)_w$, $u$ is the $u$-map for the non-trivial admissible pair $(Y,w)$, $\tau$ is the $u$-map for the 3-torus with $\tau(\rho_1)=8\rho_2$ and $\tau(\rho_2)=8\rho_1$, $\mu$ is the map on $H_\ast(SO(3))$ sending the degree 3 generator $\omega$ to the degree 0 generator $[\text{pt}]$, and the map $\varepsilon$ sends $\omega$ to itself and sends $[\text{pt}]$ to $-[\text{pt}]$.

The $E^2$-page, on the other hand, is computed from the $E^1$-page by using only the $E^1$-differential, given by the term $\partial\otimes 1\otimes \varepsilon$. Thus the $E^2$-page is given as follows:

\vspace{.15cm}
\[
    E^2 \; = \; \Big(I(Y)_w \otimes I(T^3)_t\Big)\otimes H_\ast(SO(3)).
\]
\vspace{.04cm}

\noindent We remark that $I(T^3)_t$ is equal to $C(T^3)_t$, the differentials vanishing for grading reasons. Using only that $u$ is an isomorphism on $I(Y)_w$, it is readily verified that the map sending $[\rho]\otimes \rho_1\otimes \omega$ to the element $[\rho]\in I(Y)_w$ and all other types of generators in $E^2$ to zero is an isomorphism from the kernel of the remaining differential on $E^2$ to the kernel of $u^2-64$ on $I(Y)_w$. Indeed, the inverse correspondence sends $[\rho]$ to the element $[\rho]\otimes \rho_1\otimes \omega -\frac{1}{8} u[\rho]\otimes \rho_2\otimes \omega$. A similar statement holds for the cokernel. This identifies $I(Y\# T^3)_w$ as a vector space with the mapping cone of $u^2-64$ on the group $I(Y)_w$. Modding out by a degree four involution then identifies $I^\#(Y)_w$ with the mapping cone of $u^2-64$ on the group $I(Y)_w'$.

\vspace{.50cm}

\bibliography{references}

\vspace{1.40cm}
  {\footnotesize

    \textsc{Ben-Gurion University Mathematics Department,
    Beer-Sheva, Israel}\par\nopagebreak
  \textit{E-mail address:}\;\texttt{williamb@math.bgu.ac.il}

\vspace{.75cm}

  \textsc{Simons Center for Geometry and Physics, Stony Brook, NY}\par\nopagebreak
  \textit{E-mail address:}\;\texttt{cscaduto@scgp.stonybrook.edu}
}

\end{document}